\documentclass[smallextended,numbook,runningheads]{svjour3}     
\smartqed  
\usepackage{graphicx}
\usepackage{amsmath}
\usepackage{epstopdf} 
\usepackage{mathptmx}      
\journalname{Numer. Math.}

\usepackage{amsmath}
\usepackage{stmaryrd}
\usepackage{amssymb}
\usepackage{empheq}
\usepackage{cases}
\usepackage{cite}
\usepackage{caption,graphicx}

\usepackage[dvipdfm,dvips,ps2pdf,colorlinks,pdfstartview=FitH,
bookmarksnumbered=true]{hyperref}
\hypersetup{backref=true,citebordercolor={1 1 1},citecolor=blue}
\hypersetup{linkbordercolor={1 1 1},linkcolor=blue,linktocpage=true}

\input psfig.sty

\newcommand{\eps}{\varepsilon}

\newcommand{\bR}{\mathbb{R}}
\newcommand{\bC}{\mathbb{C}}
\newcommand{\mcE}{\mathcal{E}}
\newcommand{\fe}{\mathrm{e}}
\newcommand{\sech}{\mathrm{sech}}

\newcommand{\bx}{\mathbf{x}}
\newcommand{\be}{\begin{equation}}
\newcommand{\ee}{\end{equation}}
\newcommand{\ba}{\begin{array}}
\newcommand{\ea}{\end{array}}
\newcommand{\bea}{\begin{eqnarray}}
\newcommand{\eea}{\end{eqnarray}}
\newcommand{\beas}{\begin{eqnarray*}}
\newcommand{\eeas}{\end{eqnarray*}}

\newcommand{\bv}{{\bf v}}

\numberwithin{equation}{section}

\begin{document}


\title{A uniformly accurate (UA) multiscale time integrator Fourier pseoduspectral method for
the Klein-Gordon-Schr\"{o}dinger equations in the nonrelativistic limit regime}

\titlerunning{Multiscale methods for the KGS equations}

\author{Weizhu Bao  \and  Xiaofei Zhao}

\authorrunning{W. Bao and X. Zhao} 

\institute{W. Bao  \at
              Department of Mathematics and Center for Computational
              Science and Engineering, National
              University of Singapore, Singapore 119076\\
              Fax: +65-6779-5452, Tel.: +65-6516-2765,\\
              URL: http://www.math.nus.edu.sg/\~{}bao/\\
              \email{matbaowz@nus.edu.sg}\\[1em]
           X. Zhao \at
              Department of Mathematics, National
              University of Singapore, Singapore 119076\\
              \email{zhxfnus@gmail.com}
              }

\date{Received: date / Accepted: date}

\maketitle


\begin{abstract}
A multiscale time integrator Fourier pseudospectral (MTI-FP) method is proposed and analyzed for
solving the Klein-Gordon-Schr\"{o}dinger (KGS) equations in the nonrelativistic limit regime
with a dimensionless parameter $0<\eps\le1$ which is inversely proportional to the speed of light.
In fact, the solution to the KGS equations
propagates waves with wavelength at $O(\eps^2)$ and $O(1)$ in time and space,
 respectively, when  $0<\eps\ll 1$, which brings significantly
numerical burdens in practical computation. The MTI-FP method is
designed by adapting a multiscale decomposition by
frequency to the solution at each time step and applying the Fourier pseudospectral
discretization and exponential wave integrators for spatial and temporal derivatives,
respectively. We rigorously establish  two independent error bounds for the  MTI-FP
at $O(\tau^2/\eps^2+h^{m_0})$ and $O(\eps^2+h^{m_0})$ for $\eps\in(0,1]$ with $\tau$
time step size, $h$ mesh size and $m_0\ge 4$ an integer depending on the regularity of
the solution, which
imply that the MTI-FP converges uniformly and optimally in
space with exponential convergence rate
if the solution is smooth, and uniformly in time with linear convergence rate
at $O(\tau)$ for $\eps\in(0,1]$ and optimally with quadratic convergence rate at $O(\tau^2)$
in the regime when either $\eps=O(1)$ or $0<\eps\le \tau$.
Thus the meshing strategy requirement (or $\eps$-scalability) of the MTI-FP
is $\tau=O(1)$ and $h=O(1)$ for $0<\eps\ll 1$,
which is significantly better than classical methods. Numerical results demonstrate
that our error bounds are optimal and sharp. Finally, the MTI-FP method is
applied to study numerically convergence rates of the KGS equations to the limiting models
in the nonrelativistic limit regime.


\end{abstract}


\section{Introduction}

Consider the Klein-Gordon-Schr\"{o}dinger (KGS) equations in
$d$-dimensions ($d=3,2,1$) \cite{Added,KGS0,KGS2}:
\begin{subequations}\label{dimKGS}
\begin{align}
&i\hbar\partial_t\psi(\mathbf{x},t)+\frac{\hbar^2}{2m_1}\Delta\psi(\mathbf{x},t)+
g\phi(\mathbf{x},t)\psi(\mathbf{x},t)=0,\quad \mathbf{x}\in\bR^d,\\
&\frac{1}{c^2}\partial_{tt}\phi(\mathbf{x},t)-\Delta\phi(\mathbf{x},t)+\frac{m_2^2c^2}{\hbar^2}\phi(\mathbf{x},t)
-g|\psi(\mathbf{x},t)|^2=0,
\end{align}
\end{subequations}
which represent a classical model for describing the dynamics of a complex-valued scalar
nucleon field $\psi:=\psi(\bx,t)$ interacting with a neutral real-valued scalar
meson  field $\phi:=\phi(\bx,t)$ through the Yukawa coupling with $0\ne g\in {\mathbb R}$ the coupling constant.
Here $t$ is time, $\mathbf{x}\in\bR^d$ is the spatial coordinate,
$\hbar$ is the Planck constant, $c$ is speed of light, $m_1>0$ is  the mass of a nucleon and
$m_2>0$ is the mass of a meson.

In order to scale the KGS (\ref{dimKGS}), we introduce
\be\label{scale}
\tilde{t}=\frac{t}{t_s}, \quad \tilde\bx=\frac{\bx}{x_s},
\quad \tilde\psi(\tilde \bx,\tilde t)=x_s^{d/2}\,\psi(\bx,t), \quad
\tilde \phi(\tilde \bx,\tilde t)=\frac{\phi(\bx,t)}{\phi_s},
\ee
where $x_s$, $t_s$ and $\phi_s$ are the dimensionless length unit, time unit and meson field unit, respectively,
satisfying $t_s=\frac{2m_1x_s^2}{\hbar}$ and $\phi_s=\frac{\hbar x_s^{-d/2}}{\sqrt{2m_1}}$
with $v=\frac{x_s}{t_s}=\frac{\hbar}{2m_1x_s}$ being the wave speed.
Plugging (\ref{scale}) into (\ref{dimKGS}), after a simple computation
and then removing all $\tilde{ }$, we obtain
the following dimensionless KGS equations in $d$-dimensions ($d=3,2,1$):
\begin{subequations}\label{KGS}
\begin{align}
& i\partial_t\psi(\mathbf{x},t)+\Delta\psi(\mathbf{x},t)+\lambda\phi(\mathbf{x},t)\psi(\mathbf{x},t)=0,
\quad \mathbf{x}\in\bR^d,\label{KGS 2}\\
&\eps^2\partial_{tt}\phi(\mathbf{x},t)-\Delta\phi(\mathbf{x},t)+\frac{\mu^2}
{\eps^2}\phi(\mathbf{x},t)-\lambda\left|\psi(\mathbf{x},t)\right|^2
=0,\label{KGS 1}
\end{align}
\end{subequations}
where $\varepsilon$ is a dimensionless parameter
inversely proportional to the speed of light given by
\begin{equation}\label{eps}
0<\varepsilon :=\frac{v}{c}= \frac{x_s}{ t_s\, c}=\frac{\hbar}{2cm_1x_s}\le 1,
\end{equation}
and $\mu=\frac{m_2}{2m_1}>0$ and $\lambda=\frac{g\sqrt{2m_1}x_s^{2-d/2}}{\hbar}\in {\mathbb R}$ are
two dimensionless constants which are independent of $\eps$.

We remark here that if one chooses the dimensionless length unit
$x_s=\frac{\hbar}{2cm_1}$, $ t_s=\frac{x_s}{c}=\frac{\hbar}{2c^2m_1}$
and $\phi_s=\frac{c^{d/2}\hbar^{1-d/2}}{(2m_1)^{(1-d)/2}}$
in (\ref{scale}),  then $\varepsilon =1$ in (\ref{eps}) and Eqs. (\ref{KGS}) with  $\varepsilon =1$
take the form often appearing in the 
literature \cite{Added,KGS0,KGS2}. This choice of $x_s$ is appropriate
when the wave speed is at the same order of the speed of light. However, when the wave speed is
much smaller than the speed of light, a different choice of $x_s$ is more appropriate. Note that
the choice of $x_s$ determines the observation scale of the time evolution of the particles and decides:
(i) which phenomena are `visible' by asymptotic analysis, and (ii) which phenomena
can be resolved by discretization by specified spatial/temporal grids.
In fact, there are two important parameter regimes: One is when  $\varepsilon =1$
($\Longleftrightarrow x_s=\frac{\hbar}{2cm_1}$, $t_s=\frac{\hbar}{2c^2m_1}$
and $\phi_s=\frac{c^{d/2}\hbar^{1-d/2}}{(2m_1)^{(1-d)/2}}$), then Eqs. (\ref{KGS}) describe the case
that wave speed is at the same order of the speed of light; the other one is when $0<\varepsilon \ll 1$,
then Eqs. (\ref{KGS}) are in the nonrelativistic limit regime.

To study the dynamics of the KGS (\ref{KGS}), the initial data is usually given as
\begin{equation}\label{KGS ini}
\psi(\mathbf{x},0)= \psi_0(\mathbf{x}), \quad\phi(\mathbf{x},0)= \phi_0(\mathbf{x}),\quad\partial_t\phi(\mathbf{x},0)=\frac{1}{\eps^2}\phi_1(\mathbf{x}),
\qquad \mathbf{x}\in\bR^d,
\end{equation}
where the complex-valued function $\psi_0$ and the real-valued functions
$\phi_0$ and $\phi_1$ are independent of $\eps$. The KGS equations (\ref{KGS})
are dispersive and time symmetric. They conserve the {\sl mass} of the nucleon field
\be
\|\psi(\cdot,t)\|_{L^2}^2:=\int_{\bR^d}|\psi(\bx,t)|^2\,d\bx \equiv
\int_{\bR^d}|\psi(\bx,0)|^2\,d\bx=\int_{\bR^d}|\psi_0(\bx)|^2\,d\bx,
\quad t\ge0,
\ee
and the {\sl Hamiltonian} or total {\sl energy}
\begin{eqnarray}
E(t)&:=&\int_{\bR^d}\left[\frac{1}{2}\left(\eps^2\left|\partial_t\phi\right|^2
  +|\nabla\phi|^2
  +\frac{\mu^2}{\eps^2}\left|\phi\right|^2\right)
  +|\nabla\psi|^2-\lambda|\psi|^2\phi \right]d\bx\nonumber\\
  &\equiv&\int_{\bR^d}\left[\frac{1}{2}\left(\frac{1}{\eps^2}\left|\phi_1\right|^2
  +|\nabla\phi_0|^2
  +\frac{\mu^2}{\eps^2}\left|\phi_0\right|^2\right)
  +|\nabla\psi_0|^2-\lambda|\psi_0|^2\phi_0 \right]d\bx\nonumber\\
  &=& E(0),\quad t\geq0. \label{energy}
\end{eqnarray}

For the KGS equations (\ref{KGS}) with $\eps=1$, i.e. O(1)-speed of light regime,
there are extensive  analytical and numerical results in the literatures.
For the existence and uniqueness as well as regularity,
we refer to \cite{KGS0,KGS1,KGS1.5,KGS2,KGS6,GuoBL1,GuoBL2,GuoBL3,Biler}
and references therein. For the numerical methods and comparison such as
the finite difference time domain (FDTD) methods and Crank-Nicolson Fourier pseudospectral method,  we refer to \cite{BaoYang,KGS4,KGS5,WTC} and references therein.
However, for the KGS equations (\ref{KGS}) with $0<\varepsilon\ll 1$,
i.e. nonrelativistic limit regime (or the scaled speed of light goes to infinity),
the analysis and efficient computation of the KGS equations (\ref{KGS})
are mathematically and numerically rather complicated issues.
The main difficulty is due to that the solution is highly oscillatory in time
and the corresponding energy functional $E(t)=O(\eps^{-2})$ in (\ref{energy})
becomes unbounded when $\varepsilon\to0$.

\begin{figure}[h!]
\centerline{\psfig{figure=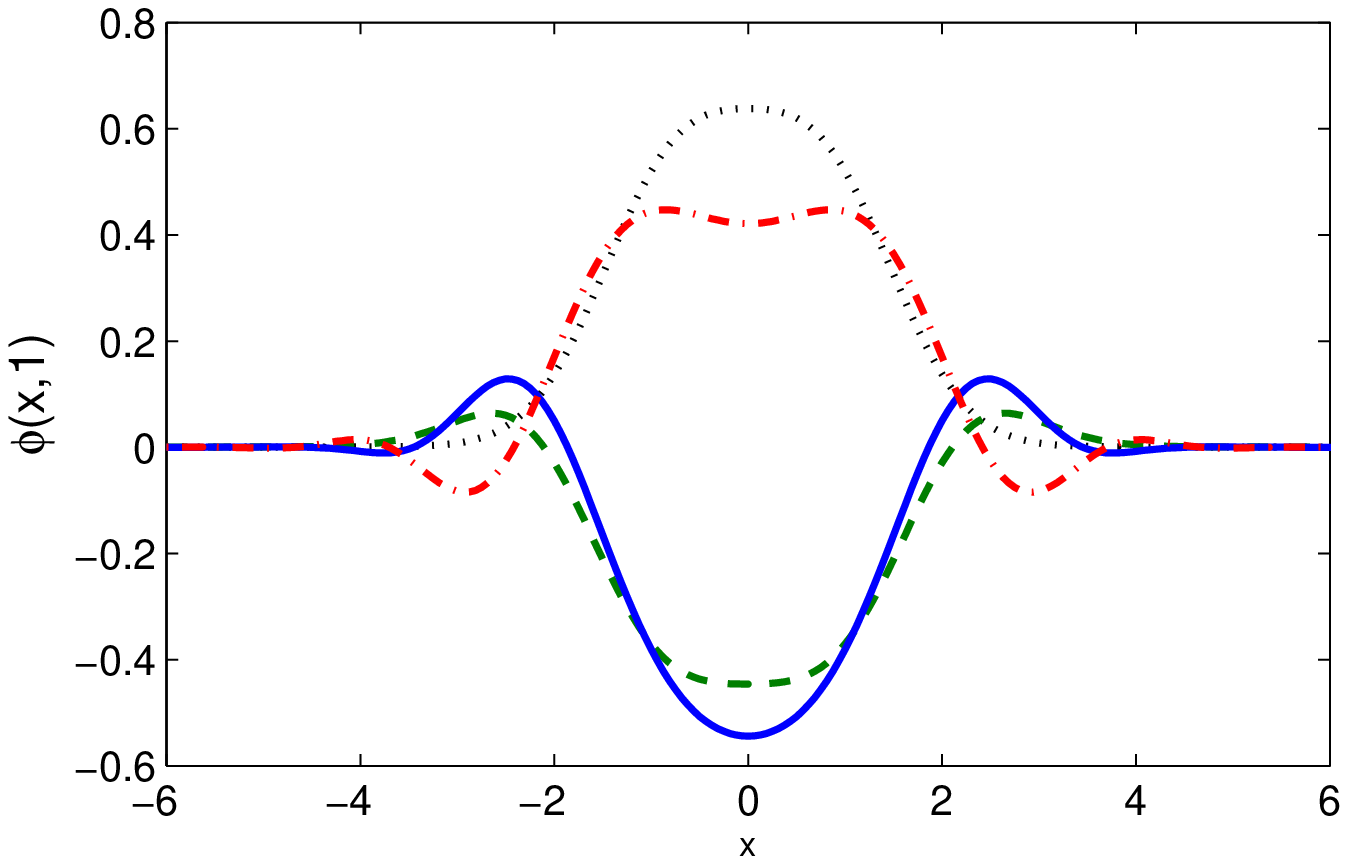,height=4.1cm,width=12cm}}
\centerline{\psfig{figure=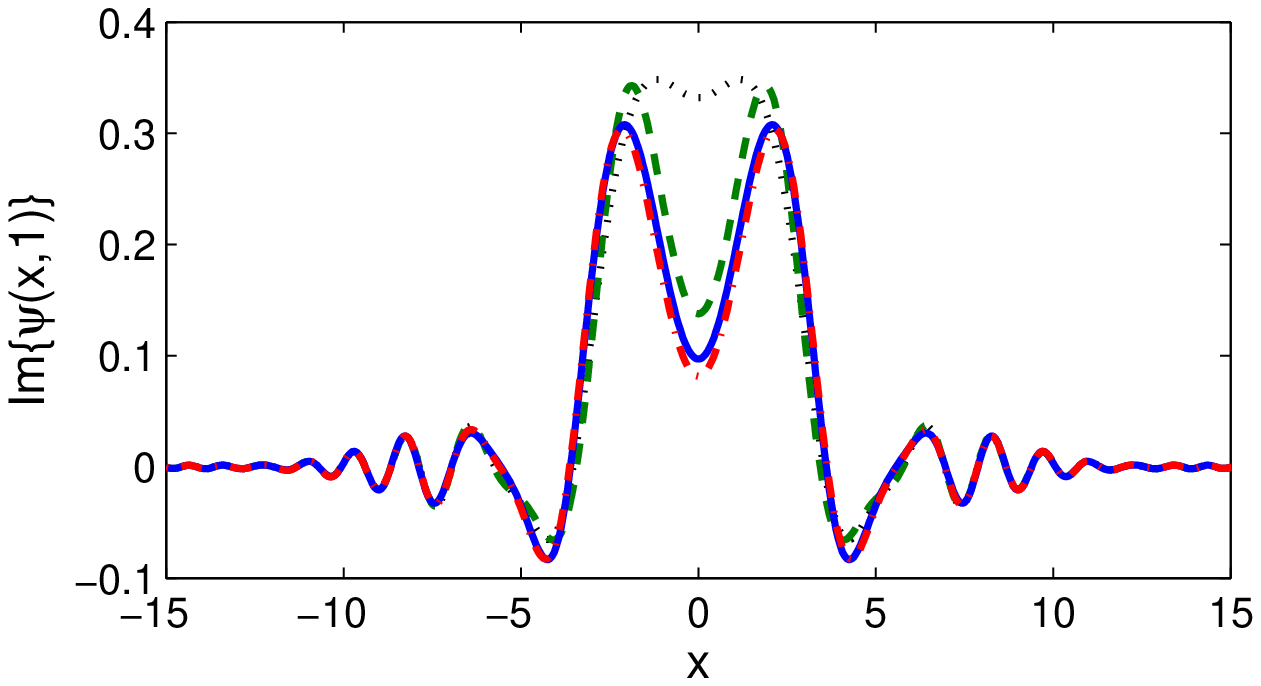,height=4.1cm,width=12cm}}
\centerline{\psfig{figure=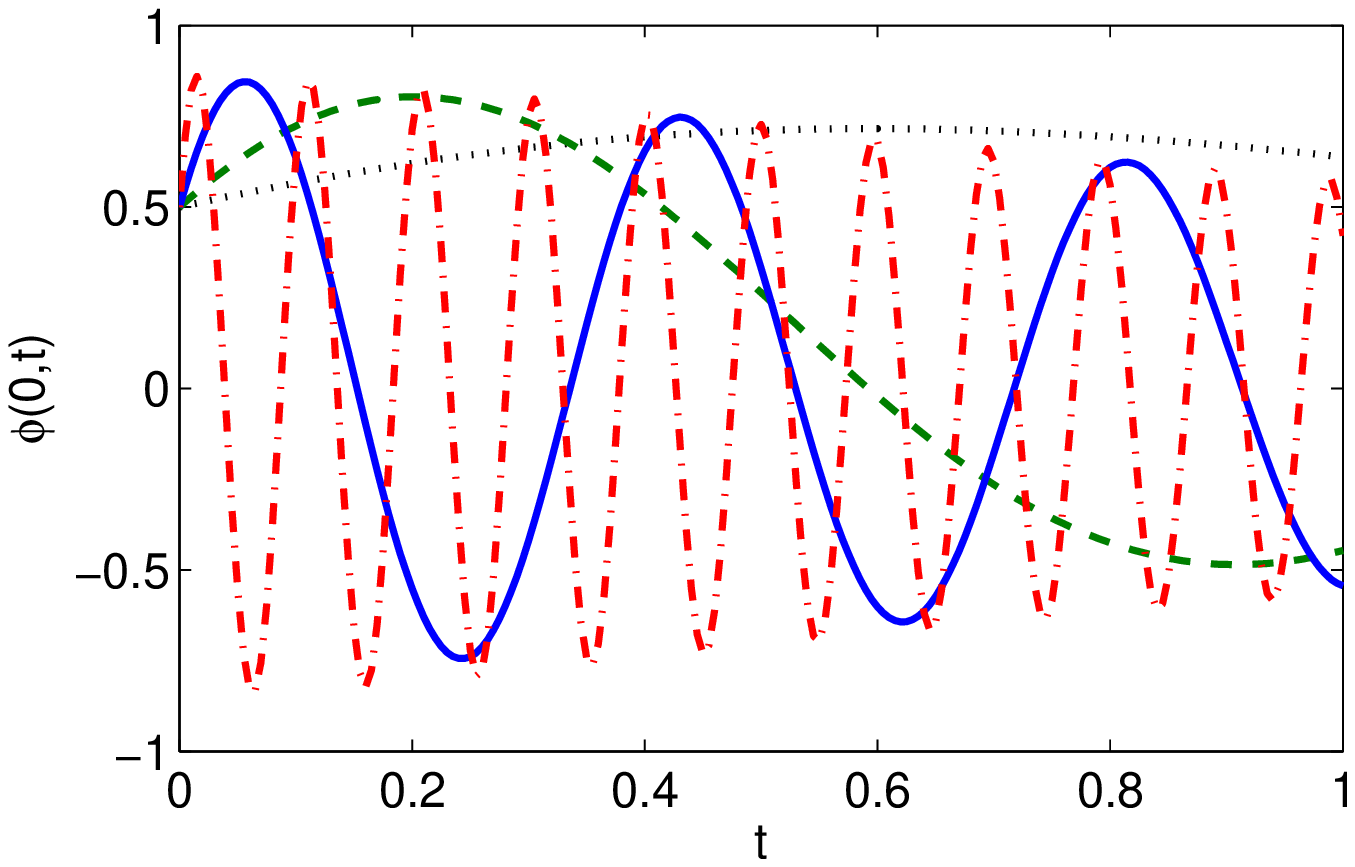,height=4.1cm,width=12cm}}
\centerline{\psfig{figure=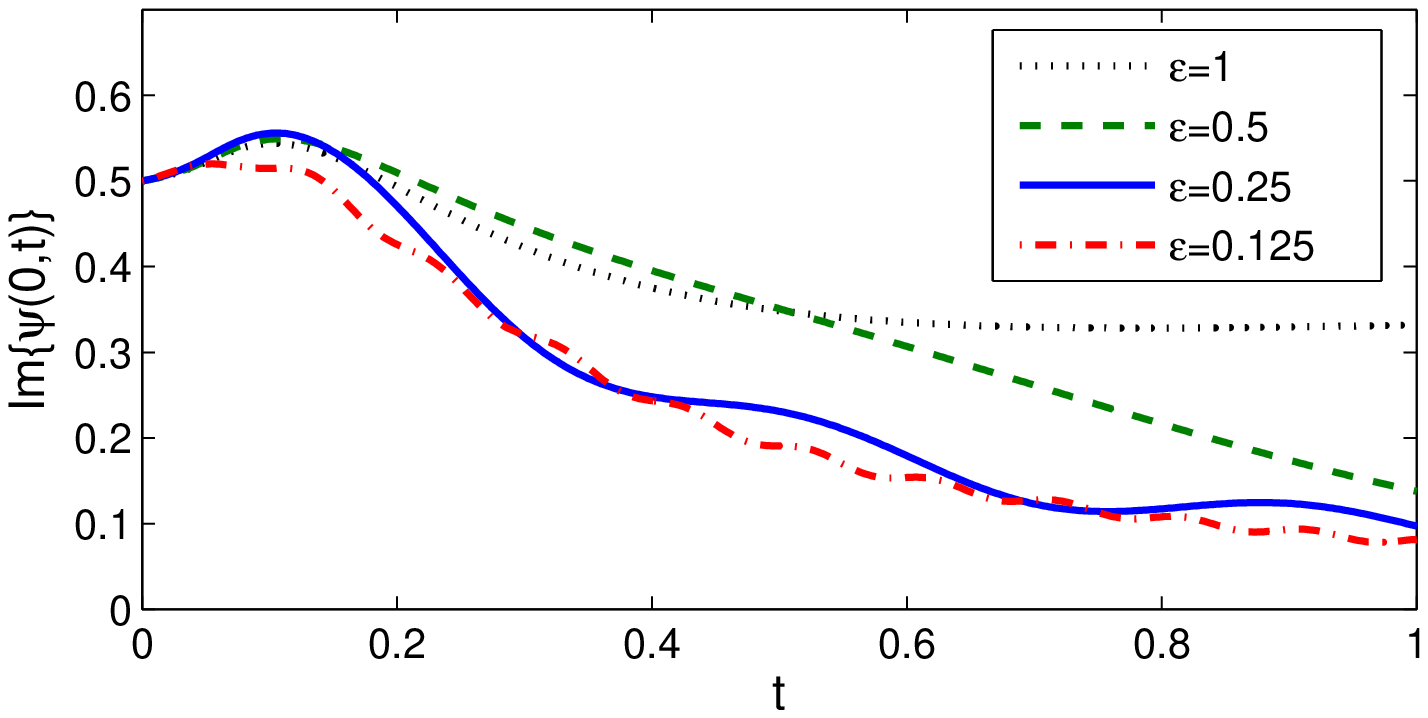,height=4.1cm,width=12cm}}
\caption{Plots of the solution of the KGS (\ref{KGS})-(\ref{KGS ini}) with $d=1$ for different $\eps$, where Im$\{\psi\}$ denotes the imaginary part of $\psi$.}\label{fig:0}
\end{figure}

Formally, in the nonrelativistic limit regime,
i.e. $0<\eps\ll1$, similarly to the analysis of the nonrelativistic limit
of the Klein-Gordon (KG) equation \cite{Machihara, Masmoudi,Masmoudi1,Lu}, taking the ansatz
\be\label{ans}
\phi(\bx,t)=\fe^{i\mu t/\eps^2}z(\bx,t)+\fe^{-i\mu t/\eps^2}\overline{z}(\bx,t)+o(\eps),
\qquad \bx\in {\mathbb R}^d, \quad t\ge0,
\ee
where $\bar{z}$ denotes the complex conjugate of a complex-valued function $z$, and plugging it into (\ref{KGS}) and (\ref{KGS ini}), we obtain the
Schr\"{o}dinger equations with wave operator as a limiting model
\begin{subequations}\label{modellimit92}
\begin{align}
&i\partial_t\psi(\bx,t)+\Delta\psi(\bx,t)=0, \qquad \bx\in\bR^d, \qquad t>0,\label{Schw1}\\
&2i\mu \partial_t z(\bx,t)+\eps^2\partial_{tt}z(\bx,t)-\Delta z(\bx,t)=0,\label{schw2}
\end{align}
\end{subequations}
with the well-prepared initial data \cite{Cai1,Cai2}
\be\label{schw45}
z(\bx,0)=\frac{1}{2}\left[\phi_0(\bx)-\frac{i}{\mu}\phi_1(\bx)\right],\
\partial_t z(\bx,0)=-\frac{i}{2\mu}\Delta z(\bx,0),\ \psi(\bx,0)=\psi_0(\bx).
\ee
In addition, after dropping the second term in (\ref{schw2}), formally we get the
Schr\"{o}dinger equations as another limiting model
\begin{subequations}\label{modellimit4}
\begin{align}
&i\partial_t\psi(\bx,t)+\Delta\psi(\bx,t)=0, \\
&2i\mu\partial_t z(\bx,t)-\Delta z(\bx,t)=0,\qquad \bx\in\bR^d, \qquad t>0,
\end{align}
\end{subequations}
with the  initial data
\be\label{schint4}
z(\bx,0)=\frac{1}{2}\left[\phi_0(\bx)-\frac{i}{\mu}\phi_1(\bx)\right],\qquad \quad\psi(\bx,0)=\psi_0(\bx),
\qquad \bx\in\bR^d.
\ee
This formally suggests that the solution of (\ref{KGS}) propagates
highly oscillatory waves with amplitude at $O(1)$ and wavelength
at $O(\eps^2)$ and $O(1)$ in time and space, respectively, when $0<\eps\ll1$.
To illustrate this, Fig. \ref{fig:0} shows the solution of the KGS equations (\ref{KGS})-(\ref{KGS ini})
with $d=1$, $\mu=\lambda=1$, $\psi_0(x)=\frac{1+i}{2}\sech(x^2)$, $\phi_0(x)=\frac{1}{2}\fe^{-x^2}$ and $\phi_1(x)=\frac{1}{\sqrt{2}}\fe^{-x^2}$ for
different $0<\eps\le 1$.

The highly temporal oscillatory nature of the solution to the KGS equations (\ref{KGS}) causes
severe burdens in practical computation, making the numerical
approximation extremely challenging and costly in the regime of
$0<\eps\ll 1$. Recently, different numerical methods
have been proposed and/or analyzed for the nonlinear
Klein-Gordon equation in the nonrelativistic limit regime
in which the solution shares similar oscillatory behavior as that
of the KGS equations (\ref{KGS}), including the FDTD
methods \cite{Dong}, exponential wave integrator Fourier pseudospectral (EWI-FP)
method \cite{Dong,BDZ}, asymptotic preserving (AP) method \cite{Faou},
stroboscopic average method (SAM) \cite{Chartier,Chartier2} and
multiscale time integrator Fourier pseudospectral (MTI-FP) method \cite{BZ}, etc.
Among them, the SAM and MTI-FP methods are uniformly convergent (UA)
for $\eps\in(0,1]$, while the FDTD, EWI-FP and AP methods are not.
From the practical computation point of view, the MTI-FP is much
simpler and thus more efficient than the SAM method.
The main aim of this paper is to propose and analyze
a MTI-FP method for the KGS equations (\ref{KGS})
in the nonrelativistic limit regime
by adapting a multiscale decomposition by
frequency to the solution at each time step and applying the Fourier pseudospectral
discretization and exponential wave integrators for spatial and temporal derivatives,
respectively. Two independent error bounds will be established for the  MTI-FP, which
imply that the MTI-FP converges uniformly and/or optimally for $0<\eps\le 1$.
The MTI-FP method is also applied to study numerically convergence rates
of the KGS equations (\ref{KGS}) to its limiting models (\ref{modellimit92})-(\ref{schw45})
and (\ref{modellimit4})-(\ref{schint4}).

The paper is organized as follows. In section \ref{sec: MD},
we introduce a multiscale decomposition for the KGS equations (\ref{KGS}) based on frequency.
A MTI-FP method is proposed in section \ref{sec:methods}, and its rigorous error bounds in energy space are established
in section \ref{sec: convergence}. Numerical results are reported in section \ref{sec: result}. Finally, some conclusions are drawn in section
\ref{sec: conclusion}. Throughout this paper,
we adopt the standard Sobolev spaces \cite{Adams} and use
the notation $A\lesssim B$ to represent that there exists a generic constant $C>0$,
which is independent of time step $\tau$ (or $n$), mesh size $h$ and $\eps$, such that $|A|\leq CB$.

\section{Multiscale decomposition}\label{sec: MD}
\label{sec:2}
Let $\tau=\Delta t>0$ be the time step size, and denote time steps by $t_n=n\tau$ for $n=0,1,\ldots$\;.
In this section, we present a multiscale decomposition
for the solution of (\ref{KGS}) on the time interval $[t_n, t_{n+1}]$
with given initial data at $t=t_n$ as
\begin{subequations}\label{Initial}
\begin{align}
&\phi(\mathbf{x},t_n)=\phi_0^n(\mathbf{x})=O(1),\qquad \partial_t\phi(\mathbf{x},t_n)=\frac{1}{\eps^2}\phi_1^n(\mathbf{x})=O\left(\frac{1}{\eps^{2}}\right),
\label{ini psi}\\
&\psi(\mathbf{x},t_n)=\psi_0^n(\mathbf{x})=O(1)\label{ini phi}.
\end{align}
\end{subequations}
Similarly to the analytical study of the nonlinear Klein-Gordon equation in the nonrelativistic limit regime
in \cite{Masmoudi,BZ}, we take an ansatz to the variable
$\phi(\mathbf{x},t):=\phi(\mathbf{x},t_n+s)$ of (\ref{KGS 1}) on the
time interval $[t_n, t_{n+1}]$  with (\ref{Initial}) as \cite{Zhao,BZ}
\begin{equation}\label{ansatz}
\phi(\mathbf{x},t_n+s)=\fe^{i\mu s/\eps^2}z^n(\mathbf{x},s)+\fe^{-i\mu s/\eps^2}\overline{z^n}(\mathbf{x},s)
+r^n(\mathbf{x},s), \quad \bx\in\bR^d,\ \ 0\leq s\leq\tau.
\end{equation}
Differentiating (\ref{ansatz}) with respect to $s$, we have
\begin{align}
\partial_s \phi(\mathbf{x},t_n+s)=&\fe^{i\mu s/\eps^2}\left[\partial_s z^n(\mathbf{x},s)+\frac{i\mu}{\eps^2}z^n(\mathbf{x},s)\right]
+\partial_sr^n(\mathbf{x},s)\nonumber\\
&+\fe^{-i\mu s/\eps^2}
\left[\partial_s\overline{ z^n}(\mathbf{x},s)-\frac{i\mu}{\eps^2}\overline{z^n}(\mathbf{x},s)\right],\quad \bx\in\bR^d,\ 0\leq s\leq\tau.\label{ansatad}
\end{align}
Plugging (\ref{ansatz}) into (\ref{KGS}), we get for $\bx\in\bR^d$, $0\leq s\leq\tau$ and $\psi(\bx,t_n+s)=:\psi^n(\bx,s)$
\beas
&&\fe^{i\mu s/\eps^2}\left[\eps^2\partial_{ss}{z}^n(\mathbf{x},s)+2i\mu\partial_s z^n(\mathbf{x},s)
-\Delta z^n(\mathbf{x},s)\right]\nonumber\\
&&+\fe^{-i\mu s/\eps^2}\left[\eps^2\partial_{ss}\overline{z^n}(\mathbf{x},s)
-2i\mu\partial_s\overline{z^n}(\mathbf{x},s)-\Delta \overline{z^n}(\mathbf{x},s)\right]\nonumber\\
&&+\eps^2\partial_{ss}r^n(\mathbf{x},s)-\Delta r^n(\mathbf{x},s)+\frac{\mu^2}{\eps^2}r^n(\mathbf{x},s)=\lambda\left|\psi^n(\mathbf{x},s)\right|^2.
\eeas
Substituting (\ref{ansatz}) into (\ref{KGS 2}) and
multiplying the above equation by $\fe^{-i\mu s/\eps^2}$ and $\fe^{i\mu s/\eps^2}$, respectively,
we can re-formulate the KGS equations (\ref{KGS}) for $\psi^n:=\psi^n(\bx,s)$,
$z^n:=z^n(\bx,s)$ and $r^n:=r^n(\bx,s)$  as
\begin{equation}\label{MDF}
\left\{
  \begin{split}
&i\partial_s\psi^n+\Delta\psi^n+\lambda\left(\fe^{\frac{i\mu s}{\eps^2}}z^n\psi^n+\fe^{-\frac{i\mu s}{\eps^2}}\overline{z^n}\psi^n+ r^n\psi^n\right)=0,\\
&2i\mu\partial_s z^n+\eps^2\partial_{ss}z^n-\Delta z^n=0, \qquad\qquad\qquad \bx\in\bR^d,\quad 0\leq s\leq\tau,\\
&\eps^2\partial_{ss}r^n-\Delta r^n+\frac{\mu^2}{\eps^2}r^n=\lambda|\psi^n|^2.
\end{split}
\right.
\ee
In order to find proper initial conditions for the system (\ref{MDF}),
setting $s=0$ in (\ref{ansatz}) and (\ref{ansatad}),
noticing (\ref{ini psi}), we obtain
\begin{numcases}
 \ z^n(\bx,0)+\overline{z^n}(\bx,0)+r^n(\bx,0)=\phi_0^n(\bx),\qquad \qquad\qquad\qquad \bx\in\bR^d,\label{init123}\\
 \frac{i\mu}{\eps^2}\left[z^n(\bx,0)-\overline{z^n}(\bx,0)\right]+\partial_s z^n(\bx,0)+\partial_s\overline{z^n}(\bx,0)
 +\partial_sr^n(\bx,0)=\frac{\phi_1^n(\bx)}{\eps^2}.\nonumber
\end{numcases}
Now we decompose the above initial data so as to: (i) equate $O\left(\frac{1}{\eps^2}\right)$ and $O(1)$ terms
in the second equation of (\ref{init123}), respectively, and (ii) be well-prepared for the second equation in (\ref{MDF})
when $0<\eps\ll 1$, i.e. $\partial_sz^n(\bx,0)$ is determined
from the second equation in (\ref{MDF}), by setting $\eps=0$ and $s=0$ \cite{Cai1,Cai2,BZ}:
\begin{equation}\label{FSW-i1}
\left\{
  \begin{split}
&z^n(\bx,0)+\overline{z^n}(\bx,0)=\phi_0^n(\bx),\qquad i\mu\left[z^n(\bx,0)-\overline{z^n}(\bx,0)\right]=\phi_1^n(\bx),\\
&2i\mu\partial_sz^n(\bx,0)-\Delta z^n(\bx,0)=0,
\qquad\qquad \qquad\qquad \qquad \bx\in\bR^d,\\
&r^n(\bx,0)=0, \qquad \partial_sr^n(\bx,0)+\partial_sz^n(\bx,0)+\partial_s\overline{z^n}(\bx,0)=0.
\end{split}
  \right.
\end{equation}
Solving (\ref{FSW-i1}) and noticing (\ref{ini phi}), we get the initial data for (\ref{MDF}) as
\begin{equation}\label{FSW-i21}
\left\{
  \begin{split}
&z^n(\bx,0)=\frac{1}{2}\left[\phi_0^n(\bx)-\frac{i}{\mu}\phi_1^n(\bx)\right],
\quad\partial_sz^n(\bx,0)=-\frac{i}{2\mu}\Delta z^n(\bx,0),\quad\bx\in\bR^d,\\
&\psi^n(\bx,0)=\psi_0^n(\bx), \quad r^n(\bx,0)=0, \quad \partial_sr^n(\bx,0)=-\partial_sz^n(\bx,0)-\partial_s\overline{z^n}(\bx,0).
\end{split}
  \right.
\end{equation}
The above decomposition can be called as multiscale decomposition by frequency (MDF).
In fact, it can also be regarded as to decompose slow waves at $\eps^2$-wavelength and fast waves at other wavelengths,
thus it can also be called as fast-slow frequency decomposition.
After solving the decomposed system (\ref{MDF}) with the initial data (\ref{FSW-i21}), we get
$\psi^n(\bx,\tau)$, $z^n(\bx,\tau)$, $\partial_s z^n(\bx,\tau)$, $r^n(\bx,\tau)$ and
 $\partial_s r^n(\bx,\tau)$.
Then we can reconstruct the solution to the KGS equations
(\ref{KGS}) at $t=t_{n+1}$ by setting $s=\tau$ in (\ref{ansatz}) and (\ref{ansatad}), i.e.
\begin{equation}\label{u n+1}
\left\{
\begin{split}
&\psi(\bx,t_{n+1})=\psi^n(\bx,\tau)=:\psi_0^{n+1}(\bx),\\
&\phi(\bx,t_{n+1})=\fe^{i\mu \tau/\eps^2}z^n(\bx,\tau)+\fe^{-i\mu\tau/\eps^2}\overline{z^n}(\bx,\tau)
+r^n(\bx,\tau)=:\phi_0^{n+1}(\bx),\\
&\partial_t \phi(\bx,t_{n+1})=\frac{1}{\eps^2}\phi_1^{n+1}(\bx),\qquad \qquad\bx\in\bR^d,\\
\end{split}\right.
\end{equation}
with
\begin{align*}
\phi_1^{n+1}(\bx):=&\fe^{i\mu\tau/\eps^2}\left[\eps^2\partial_sz^n(\bx,\tau)+i\mu z^n(\bx,\tau)\right]+\fe^{-i\mu\tau/\eps^2}
\left[\eps^2\partial_s\overline{z^n}(\bx,\tau)-i\mu\overline{z^n}(\bx,\tau)\right]\\
&+\eps^2\partial_sr^n(\bx,\tau).
\end{align*}
In summary, the MDF proceeds as a decomposition--solution-reconstruction flow at each time interval,
and this makes it essentially different from the classical modulated Fourier expansion \cite{Lubich1,Lubich2,Lubich3,Cohen,Cohen1,Cohen2,Cohen3} which only carries out
the decomposition at the initial time $t=0$.

\section{A MTI-FP method}
\label{sec:methods}

For the simplicity of notations and without loss of generality, we
take $\mu=\lambda=1$ in (\ref{KGS}) and present our numerical method in
one space dimension (1D). Generalizations to higher dimensions are
straightforward and the results remain valid. We truncate the whole-space problem
(\ref{KGS}) into a finite interval $\Omega=(a,b)$ with
periodic boundary conditions, where due to the fast decay of the solution at far field, the truncation error can be negligible by choosing $a,b$ sufficiently large. In 1D, the problem (\ref{KGS})
with $\mu=\lambda=1$ collapses to
\begin{numcases}
\, i\partial_{t}\psi(x,t) + \partial_{xx}\psi(x,t) +\phi(x,t)\psi(x,t)=0,\quad x\in\Omega,\quad t>0,\nonumber\\
\eps^2\partial_{tt}\phi(x,t)-\partial_{xx}\phi(x,t)
+ \frac{1}{\eps^2}\phi(x,t) = |\psi(x,t)|^2, \quad x\in\Omega,\ t>0,\nonumber\\
\phi(a,t)=\phi(b,t),\quad \partial_x\phi(a,t)=\partial_x\phi(b,t),\quad
\psi(a,t)=\psi(b,t)\quad t\geq 0,\label{KGS-trun}\\
\phi(x,0)=\phi_0(x),\quad \partial_t \phi(x,0)=\frac{\phi_1(x)}{\eps^2},\quad\psi(x,0)=\psi_0(x),\quad x\in\overline{\Omega}=[a,b].\nonumber
\end{numcases}
We remark that the boundary conditions considered here are
inspired by the inherent physical nature of the system and they
have been widely used in the literatures for the simulation of the KGS equations
\cite{BaoYang,BDW,KGS5,KGS4}.

Consequently, for $n\ge0$, the decomposed system MDF (\ref{MDF}) in 1D collapses to
\begin{numcases}
\ i\partial_{s}\psi^n+\partial_{xx}\psi^n+\fe^{is/\eps^2}z^n\psi^n
+\fe^{-is/\eps^2}\overline{z^n}\psi^n
+r^n\psi^n=0.\nonumber\\
2i\partial_s z^n+ \eps^2\partial_{ss}z^n-\partial_{xx}z^n=0,\quad
a<x<b, \ 0<s\leq\tau, \label{MDF trun}\\
\eps^2\partial_{ss}r^n-\partial_{xx} r^n+\frac{1}{\eps^2}r^n=\left|\psi^n\right|^2,\nonumber
\end{numcases}
The initial and boundary conditions for the above system are
\begin{equation}\label{MDF ini}
\left\{
\begin{split}
&z^n(a,s)=z^n(b,s), \quad \partial_xz^n(a,s)=\partial_xz^n(b,s),\\
&r^n(a,s)=r^n(b,s),
\quad \partial_xr^n(a,s)=\partial_xr^n(b,s), \\
& \psi^n(a,s)=\psi^n(b,s),\quad \partial_x\psi^n(a,s)=\partial_x\psi^n(b,s),\quad 0\leq s\leq\tau;\\
&z^n(x,0)=\frac{1}{2}\left[\phi_0^n(x)-i\phi_1^n(x)\right],\quad \partial_s z^n(x,0)=-\frac{i}{2}\partial_{xx} z^n(x,0),\\
&r^n(x,0)=0, \qquad \partial_s r^n(x,0)=-\partial_s z^n(x,0)-\partial_s\overline{z^n}(x,0),\\
&\psi^n(x,0)=\psi_0^n(x),\qquad a\le x\le b.
\end{split}
\right.
\end{equation}
In what follows, we present a numerical integrator
Fourier pseudospectral discretization for the MDF
(\ref{MDF trun}) with (\ref{MDF ini}), which applies the Fourier
pseudospectral discretization to spatial derivatives
followed by using some proper exponential wave integrators (EWI) for
temporal discretizations in phase (Fourier) space.

Choose the mesh size $h:=\Delta x=(b-a)/N$ with $N$ a positive integer
 and denote grid points as
$x_j:=a+jh$ for  $j=0,1,\ldots, N$. Define
\beas
&&X_N:=\mbox{span}\left\{\fe^{i\mu_l(x-a)}\ :\ x\in\overline{\Omega},\  \mu_l=\frac{2\pi l}{b-a},\ l=-\frac{N}{2},\ldots,\frac{N}{2}-1\right\},\\
&&Y_N:=\left\{\bv=(v_0,v_1,\ldots,v_N)\in\bC^{N+1}\ :\ v_0=v_N\right\},\ \|\bv\|_{l^2}=h\sum_{j=0}^{N-1}|v_j|^2.
\eeas
For a periodic function $v(x)$ on $\overline{\Omega}$ and a vector $\bv\in Y_N$,
let $P_N: L^2(\Omega)\rightarrow X_N$ be the standard $L^2$-projection operator,
 and $I_N: C(\Omega)\rightarrow X_N$ or  $Y_N \rightarrow X_N$
be the trigonometric interpolation operator \cite{book-sp-GO, ST}, i.e.
\begin{equation}\label{project oper}
(P_Nv)(x)=\sum_{l=-N/2}^{N/2-1}\widehat{v}_l\fe^{i\mu_l(x-a)},\quad (I_N\bv)(x)=\sum_{l=-N/2}^{N/2-1}\widetilde{\bv}_l\fe^{i\mu_l(x-a)},
\end{equation}
where $\widehat{v}_l$ and $\widetilde{\bv}_l$ are the Fourier and discrete Fourier transform
coefficients of the periodic function $v(x)$ and vector $\bv$, respectively, defined as
\begin{equation}\label{sine tans}
\widehat{v}_l=\frac{1}{b-a}\int_a^b v(x)\fe^{-i\mu_l(x-a)}dx,\qquad \widetilde{\bv}_l=\frac{1}{N}\sum_{j=0}^{N-1}v_j\fe^{-i\mu_l(x_j-a)}.
\end{equation}

Then a Fourier spectral method for discretizing (\ref{MDF trun}) reads:\\
Find $z_{N}^n:=z_{N}^n(x,s),\ r^n_N:=r^n_N(x,s),\ \psi^n_N:=\psi^n_N(x,s)\in X_N$ for $0\le s\leq\tau$, i.e.
\begin{equation}\left\{\begin{split}
 &z_{N}^n(x,s)=\sum_{l=-N/2}^{N/2-1}\widehat{(z_{N}^n)}_l(s)\fe^{i\mu_l(x-a)},\quad
  r_N^n(x,s)=\sum_{l=-N/2}^{N/2-1}\widehat{(r_N^n)}_l(s)\fe^{i\mu_l(x-a)},\\
 &\psi_{N}^n(x,s)=\sum_{l=-N/2}^{N/2-1}\widehat{(\psi_N^n)}_l(s)\fe^{i\mu_l(x-a)},\label{z r PN}
\end{split}\right.
\end{equation}
such that
\begin{equation}\label{MDF PN}
\left\{
\begin{split}
&i\partial_{s}\psi_N^n+\partial_{xx}\psi_N^n+\fe^{is/\eps^2}P_N(z^n_N\psi^n_N)+
\fe^{-is/\eps^2}P_N(\overline{z^n_N}\psi^n_N)
+P_N(r^n_N\psi^n_N)=0,\\
&2i\partial_s z_{N}^n+\eps^2\partial_{ss}z_{N}^n-\partial_{xx}z_{N}^n=0,\qquad 0<s\leq\tau,\quad a< x< b,\\
&\eps^2\partial_{ss}r_N^n-\partial_{xx}r_N^n+\frac{1}{\eps^2}r_N^n=P_N\left(\left|\psi_{N}^n\right|^2\right).
\end{split}
\right.
\end{equation}
Substituting (\ref{z r PN}) into (\ref{MDF PN}) and noticing the orthogonality of the Fourier basis, we get
\begin{numcases}
\ i\widehat{(\psi_N^n)}_l'(s)-\mu_l^2\widehat{(\phi_N^n)}_l(s)+\fe^{\frac{is}{\eps^2}}\widehat{(z^n_N\psi^n_N)}_l(s)
+\fe^{-\frac{is}{\eps^2}}\widehat{(\overline{z^n_N}\psi^n_N)}_l(s)+\widehat{(r^n_N\psi^n_N)}_l(s)=0,\nonumber\\
2i\widehat{(z_{N}^n)}_l'(s)+\eps^2\widehat{(z_{N}^n)}_l''(s)+\mu_l^2\widehat{(z_{N}^n)}_l(s)=0,\qquad \qquad 0<s\leq\tau,\label{MDF l} \\
\eps^2\widehat{(r_N^n)}_l''(s)+\left(\mu_l^2+\frac{1}
{\eps^2}\right)\widehat{(r_N^n)}_l(s)=\widehat{(|\psi^n_N|^2)}_l(s),\quad
-\frac{N}{2}\le l\le \frac{N}{2}-1.\nonumber
\end{numcases}
For each $-N/2\le l\le N/2-1$, we can rewrite (\ref{MDF l}) by using the variation-of-constant formula as
\begin{subequations}\label{VCF}
\begin{align}
&\widehat{(z_{N}^n)}_l(s)=a_l(s)\widehat{(z_{N}^n)}_l(0)+
\eps^2b_l(s)\widehat{(z_{N}^n)}_l'(0),\qquad \qquad 0\leq s\leq \tau,\label{VCF z}\\
&\widehat{(r_N^n)}_l(s)=\frac{\sin(\omega_l s)}{\omega_l}\widehat{(r_N^n)}_l'(0)
+\int_0^s\frac{\sin\left(\omega_l(s-\theta)\right)}
{\eps^2\omega_l}\widehat{(|\psi_N^n|^2)}_l(\theta)\,d\theta,\label{VCF r}\\
&\widehat{(\psi_N^n)}_l(s)=\fe^{-i\mu_l^2 s}\widehat{(\psi_N^n)}_l(0)+
i\fe^{-i\mu_l^2s}\int_0^s\fe^{i\left(\mu_l^2+\frac{1}{\eps^2}\right)\theta}
\widehat{(z_N^n\psi^n_N)}_l(\theta)\,d\theta\label{VCF psi}\\
&\qquad\quad\quad\ +i\fe^{-i\mu_l^2s}\int_0^s\fe^{i\left(\mu_l^2-\frac{1}{\eps^2}\right)\theta}
\widehat{(\overline{z_N^n}\psi^n_N)}_l(\theta)\,d\theta
+i\int_0^s\fe^{i\mu_l^2(\theta-s)}\widehat{(r_N^n\psi^n_N)}_l(\theta)\,d\theta,\nonumber
\end{align}
\end{subequations}
where
\begin{numcases}
\ a_l(s):=\frac{\lambda^+_l\fe^{is\lambda^-_l}-\lambda^-_l\fe^{is\lambda^+_l}}{\lambda^+_l-\lambda^-_l},\quad\ \ b_l(s):=i\frac{\fe^{is\lambda^+_l}-
\fe^{is\lambda^-_l}}{\eps^2(\lambda^-_l-\lambda^+_l)}, \quad 0\le s\le \tau,\nonumber\\
\lambda^\pm_l=-\frac{1}{\eps^2}\left(1\pm\sqrt{1+\mu_l^2\eps^2}\right),\quad\ \ \omega_l=\frac{1}{\eps^2}\sqrt{1+\mu_l^2\eps^2}.\label{nd pm def}
\end{numcases}
Differentiating (\ref{VCF z}) and (\ref{VCF r}) with respect to $s$, we obtain
\begin{subequations}\label{dz dr}
\begin{align}
&\widehat{(z_{N}^n)}_l'(s)=a_l'(s)\widehat{(z_{N}^n)}_l(0)+\eps^2b_l'(s)\widehat{(z_{N}^n)}_l'(0),\quad 0\leq s\leq \tau,\\
&\widehat{(r_N^n)}_l'(s)=\cos(\omega_l s)\widehat{(r_N^n)}_l'(0)
+\int_0^s\frac{\cos\left(\omega_l(s-\theta)\right)}{\eps^2}\widehat{(|\psi_N^n|^2)}_l(\theta)\,d\theta,
\label{rnn653}
\end{align}
\end{subequations}
where
\begin{equation*}
a_l'(s)=i\lambda^+_l\lambda^-_l\frac{\fe^{is\lambda^-_l}-\fe^{is\lambda^+_l}}{\lambda^+_l-\lambda^-_l},\quad\ b_l'(s)=\frac{\lambda^+_l\fe^{is\lambda^+_l}-
\lambda^-_l\fe^{is\lambda^-_l}}{\eps^2(\lambda^+_l-\lambda^-_l)}, \quad\ 0\le s\le \tau.
\end{equation*}
Taking $s=\tau$ in (\ref{VCF}) and (\ref{dz dr}), we immediately get
\be\label{z tau}
\begin{split}
&\widehat{(z_{N}^n)}_l(\tau)=a_l(\tau)\widehat{(z_{N}^n)}_l(0)+
\eps^2b_l(\tau)\widehat{(z_{N}^n)}_l'(0),\\
&\widehat{(z_{N}^n)}_l'(\tau)=a_l'(\tau)\widehat{(z_{N}^n)}_l(0)+
\eps^2b_l'(\tau)\widehat{(z_{N}^n)}_l'(0).
\end{split}
\ee
In order to approximate the definite integrals in (\ref{VCF r}), (\ref{VCF psi}) and (\ref{rnn653}) with $s=\tau$,
we adapt the Gautschi's type quadrature \cite{Grimm,Grimm1,Hochbruck,Hochbruck,Gaustchi,Zhao,BZ}
\[\small
\int_0^\tau\hspace{-.4mm} e^{i\delta \theta}f(\theta)\,d\theta\hspace{-.4mm}\approx\hspace{-1.5mm}
\int_0^\tau \hspace{-.4mm}e^{i\delta \theta}\left[f(0)+\theta f^\prime(0)\right]\,d\theta
=\frac{i-ie^{i\delta \tau}}{\delta} f(0)+\frac{(1-i\delta \tau)e^{i\delta \tau}-1}{\delta^2}f^\prime(0)
\]
except the last term in (\ref{VCF psi}) which is approximated
via a combination of the Gautschi's and Deuflhard's quadrature as \cite{Zhao,Deuflhard,Lubich2}
\bea
\int_0^\tau\fe^{i\mu_l^2(\theta-\tau)}
\widehat{(r_N^n\psi^n_N)}_l(\theta)\,d\theta&\approx&\int_0^\tau\fe^{i\mu_l^2(\theta-\tau)} \left[\widehat{(A_1)}_l(\theta)+\theta\; \widehat{(A_2)}_l(\theta)\right]d\theta \nonumber\\
&\approx&\frac{\tau}{2}\left[\widehat{(A_1)}_l(\tau)+\tau\, \widehat{(A_2)}_l(\tau)\right],
\eea
where $A_1(x,\theta):=r_N^n(x,\theta)\psi^n_N(x,0)$ and $A_2(x,\theta):=r_N^n(x,\theta)\partial_s\psi^n_N(x,0)$
for $0\le \theta\le \tau$ and thus $A_1(x,0)=A_2(x,0)\equiv 0$ since $r_N^n(x,0)\equiv 0$.
Thus (\ref{VCF r}), (\ref{VCF psi}) and (\ref{rnn653}) with $s=\tau$ can be approximated as
\begin{numcases}
\ \widehat{(r_N^n)}_l(\tau)\approx\frac{\sin(\omega_l \tau)}{\omega_l}\widehat{(r_N^n)}_l'(0)
+p_l(\tau)\widehat{(|\psi_N^n|^2)}_l(0)+q_l(\tau)\widehat{(|\psi_N^n|^2)}_l'(0),
\nonumber\\
\widehat{(r_N^n)}_l'(\tau)\approx\cos(\omega_l \tau)\widehat{(r_N^n)}_l'(0)
+p_l'(\tau)\widehat{(|\psi_N^n|^2)}_l(0)+q_l'(\tau)\widehat{(|\psi_N^n|^2)}_l'(0),\quad\label{z r app}\\
\widehat{(\psi_N^n)}_l(\tau)\approx\fe^{-i\mu_l^2\tau}\widehat{(\psi_N^n)}_l(0)+
c_l^+(\tau)\widehat{(z_N^n\psi^n_N)}_l(0)+d_l^+(\tau)\widehat{(z_N^n\psi^n_N)}_l'(0)\nonumber\\
\qquad\qquad\ \ +c_l^-(\tau)\widehat{(\overline{z_N^n}\psi^n_N)}_l(0)
+d_l^-(\tau)\widehat{(\overline{z_N^n}\psi^n_N)}_l'(0)
+\frac{i\tau}{2}\left[\widehat{(A_1)}_l(\tau)+\tau\, \widehat{(A_2)}_l(\tau)\right],\nonumber
\end{numcases}
where $p_l(\tau)$, $q_l(\tau)$, $p_l^\prime(\tau)$, $q_l^\prime(\tau)$, $c_l^\pm(\tau)$
and $d_l^\pm(\tau)$ are given in Appendix A.
For the derivatives, we can compute them as
\begin{align*}
&\widehat{(|\psi_N^n|^2)}_l'(0)=2\widehat{\left(\mathrm{Re}\{\overline{\psi_N^n}\partial_s\psi_N^n\}\right)}_l(0),
\quad\widehat{(z_N^n\psi^n_N)}_l'(0)=\widehat{(\partial_sz_N^n\psi^n_N)}_l(0)+\widehat{(z_N^n\partial_s\psi^n_N)}_l(0),
\end{align*}
where $\mathrm{Re}\{z\}$ denotes the real part of a complex number $z$
and for $-\frac{N}{2}\leq l\leq\frac{N}{2}-1,$
\[
\widehat{(z_N^n)}_l'(0)=\frac{i}{2}\frac{\sin(\mu_l^2\tau)}{\tau}\widehat{(z_N^n)}_l(0),\quad
\widehat{(\psi_N^n)}_l'(0)=-i\frac{\sin(\mu_l^2\tau)}{\tau}\widehat{(\psi_N^n)}_l(0)
 +i\widehat{(\phi_N^n\psi_N^n)}_l(0),
\]
which are approximations of $\partial_s z^n(x,0)=-\frac{i}{2}\partial_{xx} z^n(x,0)$ and
$\partial_s\psi^n(x,0)=i\partial_{xx}\psi^n(x,0)\\+i\phi^n(x,0)\psi^n(x,0)$, respectively \cite{BZ}.
Inserting (\ref{z tau}) and (\ref{z r app}) into (\ref{z r PN}) with setting $s=\tau$, and
noticing (\ref{u n+1}), we immediately obtain a multiscale time integrator Fourier spectral method based on the MDF (\ref{MDF trun}) for the problem (\ref{KGS-trun}).

In practice, the integrals for computing the Fourier transform
coefficients in (\ref{sine tans}), (\ref{VCF}) and (\ref{dz dr})
are usually approximated by the numerical
 quadratures \cite{Dong, BDZ,ST} given in (\ref{dz dr}). Let $\Phi_j^n$, $\dot{\Phi}_j^n$ and $\Psi_j^n$ be approximations of $\phi(x_j,t_n)$, $\partial_t\phi(x_j,t_n)$ and $\psi(x_j,t_n)$,
 respectively; $Z_{j}^{n+1},$ $\dot{Z}_{j}^{n+1},$ $R_j^{n+1}$ and $\dot{R}_j^{n+1}$ be approximations of $z^n(x_j,\tau),$ $\partial_sz^n(x_j,\tau),$ $r^n(x_j,\tau)$ and $\partial_s r^n(x_j,\tau)$, respectively,
 for $j=0,1,\ldots,N$. Choosing $\Phi_j^0=\phi_0(x_j)$, $\dot{\Phi}_j^0=\phi_1(x_j)/\eps^2$ and $\Psi_j^0=\psi_0(x_j)$ for
 $0\le j\le N$, 
 a multiscale time integrator Fourier pseudospectral (MTI-FP) discretization for the problem (\ref{KGS-trun}) reads  for $n\ge0,$
\begin{numcases}
\,\Phi^{n+1}_j=\fe^{i\tau/\eps^2}Z_{j}^{n+1}+\fe^{-i\tau/\eps^2}
\overline{Z_{j}^{n+1}}+R^{n+1}_j,\quad  j=0,1,\ldots,N,\nonumber\\
\dot{\Phi}^{n+1}_j=\fe^{i\tau/\eps^2}\left(\dot{Z}_{j}^{n+1}+\frac{i}{\eps^2}Z_{j}^{n+1}\right)
+\fe^{-i\tau/\eps^2}\left(\overline{\dot{Z}_{j}^{n+1}}-\frac{i}{\eps^2}
\overline{Z_{j}^{n+1}}\right)+\dot{R}^{n+1}_j,\nonumber\\
\Psi_{j}^{n+1}=\sum_{l=-\frac{N}{2}}^{\frac{N}{2}-1}\widetilde{(\Psi^{n+1})}_l\fe^{i\mu_l(x_j-a)},
\label{MTI-FP S}
\end{numcases}
where
\begin{numcases}
\,Z_{j}^{n+1}=\sum_{l=-\frac{N}{2}}^{\frac{N}{2}-1}\widetilde{(Z^{n+1})}_l\fe^{i\mu_l(x_j-a)},\quad
R_{j}^{n+1}=\sum_{l=-\frac{N}{2}}^{\frac{N}{2}-1}\widetilde{(R^{n+1})}_l\fe^{i\mu_l(x_j-a)},\nonumber\\
\dot{Z}_{j}^{n+1}=\sum_{l=-\frac{N}{2}}^{\frac{N}{2}-1}\widetilde{(\dot{Z}^{n+1})}_l\fe^{i\mu_l(x_j-a)},\quad
\dot{R}_{j}^{n+1}=\sum_{l=-\frac{N}{2}}^{\frac{N}{2}-1}\widetilde{(\dot{R}^{n+1})}_l\fe^{i\mu_l(x_j-a)},\nonumber
\end{numcases}
and for $-N/2\le l\le N/2-1$,
\begin{subequations} \label{sine psu coeff}
\begin{align}\small
\widetilde{(Z^{n+1})}_l\hspace{-1mm}=&a_l(\tau)\widetilde{(Z^{0})}_l+\eps^2b_l(\tau)\widetilde{(\dot{Z}^{0})}_l,\,
\widetilde{(\dot{Z}^{n+1})}_l\hspace{-1mm}=a_l'(\tau)\widetilde{(Z^0)}_l+\eps^2b_l'(\tau)\widetilde{(\dot{Z}^0)}_l,\label{sine psu coeff z}\\
\widetilde{(R^{n+1})}_l=&\frac{\sin(\omega_l \tau)}{\omega_l}\widetilde{(\dot{R}^{0})}_l+p_l(\tau)\widetilde{(|\Psi^n|^2)}_l
+2q_l(\tau)\widetilde{(\mathrm{Re}\{\overline{\Psi^n}\dot{\Psi}^n\})}_l,\label{sine psu coeff r}\\
\widetilde{(\dot{R}^{n+1})}_l=&\cos(\omega_l \tau)\widetilde{(\dot{R}^0)}_l+p_l'(\tau)\widetilde{(|\Psi^n|^2)}_l
+2q_l'(\tau)\widetilde{(\mathrm{Re}\{\overline{\Psi^n}\dot{\Psi}^n\})}_l,\label{sine psu coeff dr}\\
\widetilde{(\Psi^{n+1})}_l=&\fe^{-i\mu_l^2\tau}\widetilde{(\Psi^n)}_l+
c_l^+(\tau)\widetilde{(Z^0\Psi^n)}_l+d_l^+(\tau)\left[\widetilde{(\dot{Z}^0\Psi^n)}_l+
\widetilde{(Z^0\dot{\Psi}^n)}_l\right]
\nonumber\\
&+c_l^-(\tau)\widetilde{(\overline{Z^0}\Psi^n)}_l+d_l^-(\tau)\left[\widetilde{(\overline{\dot{Z}^0}\Psi^n)}_l
+\widetilde{(\overline{Z^0}\dot{\Psi}^n)}_l\right]
+\frac{i\tau}{2}\widetilde{(R^{n+1}\Psi^n)}_l\nonumber\\
&+\frac{i\tau^2}{2}\widetilde{(R^{n+1}\dot{\Psi}^n)}_l,\label{sine psu coeff psi}
\end{align}
\end{subequations}
with
\begin{subequations}\label{MTI-FP E}
\begin{align}
&\widetilde{\left(Z^{0}\right)}_{l}=\frac{1}{2}\left[\widetilde{\left(\Phi^n\right)}_l
-i\eps^2\widetilde{\left(\dot{\Phi}^n\right)}_l\right],\quad
\widetilde{\left(\dot{Z}^0\right)}_{l}=\frac{i\sin\left(\mu_l^2\tau\right)}{2\tau}\widetilde{(Z^0)}_l,\\
&\widetilde{\left(\dot{R}^0\right)}_l=-\widetilde{\left(\dot{Z}^0\right)}_{l}
-\widetilde{\left(\overline{\dot{Z}^0}\right)}_{l},\quad \widetilde{(\dot{\Psi}^n)}_l
=-\frac{i\sin(\mu_l^2\tau)}{\tau}\widetilde{(\Psi^n)}_l
+i\widetilde{(\Phi^n\Psi^n)}_l.\label{MTI-SP E1}
\end{align}
\end{subequations}
This MTI-FP method for the KGS equations (\ref{KGS-trun}) (or (\ref{KGS})) is explicit, accurate,
easy to implement and very efficient due to the  discrete fast Fourier transform. The memory
cost is $O(N)$ and the computational cost per time step is $O(N \log N )$.

\section{Uniform convergence of MTI-FP} \label{sec: convergence}
In this section, we establish an error bound
for the MTI-FP method (\ref{MTI-FP S}) of the KGS (\ref{KGS-trun}), which is uniformly
for $\eps\in (0,1]$. Let $0<T< T^*$ with $T^*$ the maximum existence time of the solution
of the problem (\ref{KGS-trun}), motivated by the formal asymptotic results (\ref{ans}) and (\ref{modellimit92})
as well as the numerical results (cf. Fig. 1.1), we make the following assumptions
on the solution of the problem (\ref{KGS-trun}) -- there exists an integer $m_0\ge 4$
such that $\phi\in C^1([0,T];H_p^{m_0+4}(\Omega))$, $\psi\in
C([0,T];H_p^{m_0+2}(\Omega))\cap C^1([0,T];H_p^{m_0}(\Omega))\cap
C^2([0,T];H_p^{m_0-2}(\Omega))$ and
\be\label{assumption}
\begin{split}
&\left\|\phi\right\|_{L^\infty([0,T]; H^{m_0+4})}+\eps^2 \left\|\partial_t\phi\right\|_{L^\infty([0,T]; H^{m_0+4})}\lesssim 1,\\
&\left\|\psi\right\|_{L^\infty([0,T]; H^{m_0+2})}+\left\|\partial_t\psi\right\|_{L^\infty([0,T]; H^{m_0})}+
\eps^2\left\|\partial_{tt}\psi\right\|_{L^\infty([0,T]; H^{m_0-2})}\lesssim 1,
\end{split}
\ee
where $H_p^m(\Omega)=\left\{f(x)\in H^m(\Omega)\ |\ f^{(k)}(a)=f^{(k)}(b), k=0,1,\ldots,m-1\right\}\subset H^m(\Omega)$.
From the first equation in (\ref{KGS-trun}), i.e. the Schr\"{o}dinger equation, it is easy to see that
$$i\partial_t\rho(x,t) +\partial_x\left[\overline{\psi}\partial_x\psi(x,t)-
\psi\partial_x\overline{\psi}(x,t)\right]=0,\quad x\in\Omega,\ t>0,$$
where $\rho(x,t):=|\psi(x,t)|^2$.
Then under the assumption (\ref{assumption}), we have
\begin{equation}\label{abs psi}
\begin{split}
&\rho \in C\left([0,T];H_p^{m_0+2}(\Omega)\right)\cap C^1\left([0,T];H_p^{m_0}(\Omega)\right)\cap
C^2\left([0,T];H_p^{m_0-2}(\Omega)\right),\\
&\|\partial_t^k\rho\|_{L^\infty([0,T];H^{m_0+2-2k}(\Omega))}\lesssim1,\quad k=0,1,2.
\end{split}
\end{equation}
Denote $C_\phi=\max_{0<\eps\le 1}\left\{\left\|\phi\right\|_{L^\infty([0,T];H^{m_0+4}(\Omega))},
\eps^2\left\|\partial_t\phi\right\|_{L^\infty([0,T];H^{m_0+4}(\Omega))}\right\}$ and
$C_\psi=\max_{0<\eps\le 1}\left\{\left\|\psi\right\|_{L^\infty([0,T];H^{m_0+2}(\Omega))}\right\}$.

Let $\Phi^n=(\Phi_0^n,\Phi_1^n,\ldots,\Phi_{N}^n)\in Y_N$, $\Psi^n=(\Psi_0^n,\Psi_1^n,\ldots,\Psi_{N}^n)\in Y_N$
and $\dot{\Phi}^n=(\dot{\Phi}_0^n$, $\dot{\Phi}_1^n,\ldots,\dot{\Phi}_{N}^n)\in Y_N$ ($n\ge0$) are the numerical solution
obtained from the MTI-FP method (\ref{MTI-FP S})-(\ref{MTI-FP E}), denote
their interpolations as
\be
\phi_I^n(x):=(I_N\Phi^n)(x),\quad\psi_I^n(x):=(I_N\Psi^n)(x),\quad
\dot{\phi}_I^n(x):=(I_N\dot{\Phi}^n)(x),\quad x\in\overline{\Omega},
\ee
and define the error functions as
\be\label{error fun}
\begin{split}
&e_\phi^n(x):=\phi(x,t_n)-\phi_I^n(x),\qquad
e_\psi^n(x):=\psi(x,t_n)-\psi_I^n(x),\\
&\dot{e}_\phi^n(x):=\partial_t\phi(x,t_n)-\dot{\phi}_I^n(x),\qquad
 x\in\overline{\Omega},\qquad 0\leq n\leq\frac{T}{\tau};
\end{split}
\ee
then we have the following error estimates for the MTI-FP method (\ref{MTI-FP S})-(\ref{MTI-FP E}).

\begin{theorem}[Error bounds of MTI-FP]\label{main thm}
Under the assumption (\ref{assumption}),
there exist two constants $0<h_0\leq1$ and $0<\tau_0\leq1$ sufficiently small and
independent of $\eps$
such that for any $0<\eps\leq1$, when $0<h\leq h_0$ and $0<\tau\leq\tau_0$, we have
\begin{align}
&\left\|e_\phi^n\right\|_{H^2}+\left\|e_\psi^n\right\|_{H^2}+\eps^2\left\|\dot{e}_\phi^n\right\|_{H^2}\lesssim h^{m_0}
+\frac{\tau^2}{\eps^2},
\label{MTI error bound1}\\
&\left\|e_\phi^n\right\|_{H^2}+\left\|e_\psi^n\right\|_{H^2}+\eps^2\left\|\dot{e}_\phi^n\right\|_{H^2}\lesssim h^{m_0}
+\eps^2, \qquad 0\leq n\leq\frac{T}{\tau},\label{MTI error bound2}\\
&\left\|\phi^n_I\right\|_{H^2}\leq C_\phi+1,\qquad \left\|\psi^n_I\right\|_{H^2}\leq C_\psi+1,
\qquad\left\|\dot{\phi}^n_I\right\|_{H^2}\leq \frac{C_\phi+1}{\eps^2}.
\label{MTI sol bound}
\end{align}
Thus, by taking the minimum of the two error bounds (\ref{MTI error bound1}) and (\ref{MTI error bound2})
for $\eps\in(0,1]$, we obtain a uniform error bound with respect to $\eps\in(0,1]$ for $0\le n\le \frac{T}{\tau}$
\begin{equation}\label{uniform bound}
\left\|e_\phi^n\right\|_{H^2}+\left\|e_\psi^n\right\|_{H^2}+\eps^2\left\|\dot{e}_\phi^n\right\|_{H^2}\lesssim h^{m_0}+\min_{0<\eps\leq1}\left\{\frac{\tau^2}{\eps^2},\eps^2\right\}\lesssim h^{m_0}+\tau.
\end{equation}
\end{theorem}


In order to prove the above theorem,
 we introduce
\be
\begin{split}
&e^{n}_{\phi,N}(x):=(P_N\phi)(x,t_n)-\phi_I^n(x),\ \ \dot{e}^{n}_{\phi,N}(x):=P_N(\partial_{t}\phi)(x,t_n)-\dot{\phi}_I^n(x),\label{e_N}\\
&e^{n}_{\psi,N}(x):=(P_N\psi)(x,t_n)-\psi_I^n(x),\qquad x\in\overline{\Omega},\qquad 0\leq n\leq\frac{T}{\tau}.
\end{split}
\ee
Using the triangle inequality and noticing the assumption (\ref{assumption}), we have
\be\label{ephin896}
\begin{split}
&\|e_\phi^n\|_{H^2}\le \|\phi(x,t_n)-(P_N\phi)(x,t_n)\|_{H^2}+\|e_{\phi,N}^n\|_{H^2}\lesssim
h^{m_0+2}+\|e_{\phi,N}^n\|_{H^2},\\
&\|e_\psi^n\|_{H^2}\le \|\psi(x,t_n)-(P_N\psi)(x,t_n)\|_{H^2}+\|e_{\psi,N}^n\|_{H^2}\lesssim
h^{m_0}+\|e_{\psi,N}^n\|_{H^2},\\
&\|\dot{e}_\phi^n\|_{H^2}\le \|\dot{\phi}(x,t_n)-(P_N\dot{\phi})(x,t_n)\|_{H^2}+\|\dot{e}_{\phi,N}^n\|_{H^2}\lesssim
\frac{h^{m_0+2}}{\eps^2}+\|\dot{e}_{\phi,N}^n\|_{H^2}.
\end{split}
\ee
Thus we need only obtain estimates for $\|e_{\phi,N}^n\|_{H^2}$, $\|e_{\psi,N}^n\|_{H^2}$ and
$\|\dot{e}_{\phi,N}^n\|_{H^2}$, which will be done by introducing the error energy functional
\be
\mathcal{E}\left(e_{\phi,N}^n,e_{\psi,N}^n,\dot{e}_{\phi,N}^n\right):=
\eps^2\left\|\dot{e}_{\phi,N}^n\right\|_{H^2}^2+\left\|\partial_x e_{\phi,N}^n\right\|_{H^2}^2+\frac{1}{\eps^2}\left(\left\|e_{\phi,N}^n\right\|_{H^2}^2
+\left\|e_{\psi,N}^n\right\|_{H^2}^2\right),
\label{error engery}
\ee
and establishing the following several lemmas.

For any $\bv\in Y_N$, denote $v_{-1}=v_{N-1}$ and $v_{N+1}=v_1$, and
define the difference operators $\delta_x^+\bv\in Y_N$ and $\delta_x^2\bv\in Y_N$ as
\[\delta_x^+v_j=\frac{v_{j+1}-v_j}{h}, \qquad \delta_x^2v_j=\frac{v_{j+1}-2v_j+v_{j-1}}{h^2},
\qquad j=0,1,\ldots,N,\]
and the  norms as
\be
\|\bv\|_{Y,1}^2=\|\bv\|_{l^2}^2+\|\delta_x^+\bv\|_{l^2}^2,
\qquad
\|\bv\|_{Y,2}^2=\|\bv\|_{l^2}^2+\|\delta_x^+\bv\|_{l^2}^2+\|\delta_x^2\bv\|_{l^2}^2,
\ee
with
\[\|\bv\|_{l^2}^2=h\sum_{j=0}^{N-1}|v_j|^2,\ \ \|\delta_x^+\bv\|_{l^2}^2=h\sum_{j=0}^{N-1}|\delta_x^+v_j|^2,\ \
\|\delta_x^2\bv\|_{l^2}^2=h\sum_{j=0}^{N-1}|\delta_x^2v_j|^2,\]
then we can have the following estimates \cite{BZ,Cai2}
\begin{eqnarray}\label{normequ986}
 \|I_N\bv\|_{H^1}\lesssim \|\bv\|_{Y,1}\lesssim \|I_N\bv\|_{H^1},
\ \ \|I_N\bv\|_{H^2}\lesssim \|\bv\|_{Y,2}\lesssim \|I_N\bv\|_{H^2},
\quad \forall \bv\in Y_N.
\end{eqnarray}


\begin{lemma}[A new formulation of MTI-FP]\label{lm: rewrite scheme}
For the numerical solution $\phi_I^{n+1}$ and $\dot{\phi}_I^{n+1}$ ($n\ge0$) obtained
from the MTI-FP method (\ref{MTI-FP S})-(\ref{MTI-FP E}), we have
\be\label{observ of scheme}
\phi_I^{n+1}(x)=\sum_{l=-N/2}^{N/2-1}\widetilde{(\Phi^{n+1})}_l\fe^{i\mu_l(x-a)},\qquad
\dot{\phi}_I^{n+1}(x)=\sum_{l=-N/2}^{N/2-1}\widetilde{(\dot{\Phi}^{n+1})}_l\fe^{i\mu_l(x-a)},
\ee
where
\[
\begin{split}
\widetilde{(\Phi^{n+1})}_l=&\cos(\omega_l\tau)\widetilde{(\Phi^{n})}_l
+\frac{\sin(\omega_l\tau)}{\omega_l}\widetilde{(\dot{\Phi}^{n})}_l+p_l(\tau)\widetilde{(|\Psi^n|^2)}_l
+2q_l(\tau)\widetilde{(\mathrm{Re}\{\overline{\Psi^n}\dot{\Psi}^n\})}_l,\\
\widetilde{(\dot{\Phi}^{n+1})}_l=&\cos(\omega_l\tau)\widetilde{(\dot{\Phi}^{n})}_l\hspace{-0.5mm}
-\omega_l\sin(\omega_l\tau)\widetilde{(\Phi^{n})}_l\hspace{-0.5mm}
+p_l'(\tau)\widetilde{(|\Psi^n|^2)}_l\hspace{-0.5mm}+\hspace{-0.5mm}2q_l'(\tau)\widetilde{(\mathrm{Re}\{\overline{\Psi^n}\dot{\Psi}^n\})}_l.
\end{split}
\]
\end{lemma}

\begin{proof}
For $l=-\frac{N}{2},\ldots,\frac{N}{2}-1$, from (\ref{MTI-FP S}), we get
\begin{numcases}
\,\widetilde{(\Phi^{n+1})}_l=\fe^{\frac{i\tau}{\eps^2}}\widetilde{(Z^{n+1})}_l
+\fe^{-\frac{i\tau}{\eps^2}}\widetilde{\left(\overline{Z^{n+1}}\right)}_l
+\widetilde{(R^{n+1})}_l,\label{lm: rewrite scheme eq0}\\
\widetilde{(\dot{\Phi}^{n+1})}_l=\fe^{\frac{i\tau}{\eps^2}}\Big[\widetilde{(\dot{Z}^{n+1})}_l\hspace{-0.5mm}
+\frac{i}{\eps^2}\widetilde{(Z^{n+1})}_l \Big]\hspace{-0.5mm}
+\fe^{-\frac{i\tau}{\eps^2}}\Big[\widetilde{\left(\overline{\dot{Z}^{n+1}}\right)}_l\hspace{-0.5mm}
-\frac{i}{\eps^2}\widetilde{\left(\overline{Z^{n+1}}\right)}_l \Big]
\hspace{-0.5mm}+\widetilde{(\dot{R}^{n+1})}_l,\nonumber
\end{numcases}
and from (\ref{MTI-FP E}), we obtain
\begin{equation}\label{lm: rewrite scheme eq1}
\begin{split}
&\widetilde{(Z^0)}_l=\frac{1}{2}\left[\widetilde{(\Phi^n)}_l-i\eps^2\widetilde{(\dot{\Phi}^n)}_l\right],\quad
\widetilde{(\dot{R}^0)}_l=-\widetilde{(\dot{Z}^0)}_l-\widetilde{\left(\overline{\dot{Z}^0}\right)}_l.
\end{split}
\end{equation}
Noticing (\ref{nd pm def}) and  (\ref{project oper}), for $\bv\in Y_N$, we have
\[a_l(\tau)=a_{-l}(\tau),\ \ b_l(\tau)=b_{-l}(\tau),\ \ l=-\frac{N}{2},\ldots,\frac{N}{2}-1,\quad
\widetilde{\left(\overline{\bv}\right)}_l=\left\{
\begin{split}
&\overline{\widetilde{\bv}_{-l}},\ \ |l|\leq\frac{N}{2}-1,\\
&\overline{\widetilde{\bv}_{l}},\ \ l=-\frac{N}{2}.
\end{split}\right.
\]
Plugging (\ref{lm: rewrite scheme eq1}) into (\ref{sine psu coeff z})-(\ref{sine psu coeff dr}),
and  (\ref{sine psu coeff}) into (\ref{lm: rewrite scheme eq0}), we get
\begin{align*}
\widetilde{(\Phi^{n+1})}_l=&\mathrm{Re}\left\{\fe^{\frac{i\tau}{\eps^2}}a_l(\tau)\right\}\widetilde{(\Phi^{n})}_l
+\eps^2\mathrm{Im}\left\{\fe^{\frac{i\tau}{\eps^2}}a_l(\tau)\right\}\widetilde{(\dot{\Phi}^{n})}_l
+\eps^2\fe^{\frac{i\tau}{\eps^2}}b_l(\tau)\widetilde{(\dot{Z}^0)}_l\\
&+\eps^2\fe^{-\frac{i\tau}{\eps^2}}\overline{b_l}(\tau)\widetilde{\left(\overline{\dot{Z}^0}\right)}_l+
\frac{\sin(\omega_l\tau)}{\omega_l}\widetilde{(\dot{R}^0)}_l+p_l(\tau)\widetilde{(|\Psi^n|^2)}_l\\
&+2q_l(\tau)\widetilde{(\mathrm{Re}\{\overline{\Psi^n}\dot{\Psi}^n\})}_l,\\
\widetilde{(\dot{\Phi}^{n+1})}_l=&\eps^2\mathrm{Im}\left\{\fe^{\frac{i\tau}{\eps^2}}a_l'(\tau)\hspace{-0.5mm}
+\frac{i}{\eps^2}\fe^{\frac{i\tau}{\eps^2}}a_l(\tau)\right\}\widetilde{(\dot{\Phi}^{n})}_l\hspace{-0.5mm}
+\eps^2\fe^{-\frac{i\tau}{\eps^2}}\left[\overline{b_l'}(\tau)\hspace{-0.5mm}-\frac{i}{\eps^2}\overline{b_l}(\tau)\right]
\widetilde{\left(\overline{\dot{Z}^0}\right)}_l\\
&+\mathrm{Re}\left\{\fe^{\frac{i\tau}{\eps^2}}a_l'(\tau)
+\frac{i}{\eps^2}\fe^{\frac{i\tau}{\eps^2}}a_l(\tau)\right\}\widetilde{(\Phi^{n})}_l
+\eps^2\fe^{\frac{i\tau}{\eps^2}}\left[b_l'(\tau)+\frac{i}{\eps^2}b_l(\tau)\right]\widetilde{(\dot{Z}^0)}_l\\
&+\cos(\omega_l\tau)\widetilde{(\dot{R}^0)}_l+p_l'(\tau)\widetilde{(|\Psi^n|^2)}_l+2q_l'(\tau)
\widetilde{(\mathrm{Re}\{\overline{\Psi^n}\dot{\Psi}^n\})}_l,
\end{align*}
where $\mathrm{Im}\{z\}$ denotes the imaginary part of a complex number $z$.
Then (\ref{observ of scheme}) can be obtained by combining the above equalities and (\ref{nd pm def}).
\qed
\end{proof}

For the solution $\phi^n(x,s):=\phi(x,t_n+s),\psi^n(x,s):=\psi(x,t_n+s)$ of (\ref{KGS-trun}), we have
\begin{subequations}\label{KG fft eq}
\begin{align}
&\eps^2\widehat{\phi}_l''(t_n+s)+\mu_l^2\widehat{\phi}_l(t_n+s)+\frac{1}{\eps^2}\widehat{\phi}_l(t_n+s)=
\widehat{(|\psi|^2)}_l(t_n+s),\quad s>0\\
&i\widehat{\psi}_l'(t_n+s)-\mu_l^2\widehat{\phi}_l(t_n+s)+\widehat{(\phi\psi)}_l(t_n+s)=0,\quad l\in{\mathbb Z}.
\end{align}
\end{subequations}
By applying the variation-of-constant formula to (\ref{KG fft eq})
and noticing the ansatz (\ref{ansatz}) with $\rho^n(x,s):=|\psi(x,t_n+s)|^2$,
we get
\begin{subequations}\label{vcf: u cubic}
\begin{align}
 \widehat{\phi}_l(t_{n+1})=&\cos(\omega_l\tau)\widehat{\phi}_l(t_n)+
 \frac{\sin(\omega_l\tau)}{\omega_l}\widehat{\phi}_l'(t_n)
 +\int_0^\tau\frac{\sin(\omega_l(\tau-\theta))}{\eps^2\omega_l}\widehat{(\rho^n)}_l(\theta)d\theta,
 \qquad\label{vcf: u cubic a}\\
\widehat{\phi}_l'(t_{n+1})=&\cos(\omega_l\tau)\widehat{\phi}_l'(t_n) -\omega_l\sin(\omega_l\tau)\widehat{\phi}_l(t_n) +\int_0^\tau\frac{\cos(\omega_l(\tau-\theta))}{\eps^2}\widehat{(\rho^n)}_l(\theta)d\theta,\label{vcf: u cubic b}\\
\widehat{\psi}_l(t_{n+1})=&\fe^{-i\mu_l^2\tau}\widehat{\psi}_l(t_n)+
i\fe^{-i\mu_l^2\tau}\int_0^\tau\fe^{i\left(\mu_l^2+\frac{1}{\eps^2}\right)\theta}
\widehat{(z^n\psi^n)}_l(\theta)d\theta\label{vcf: u cubic c}\\
&+i\fe^{-i\mu_l^2\tau}\int_0^\tau\fe^{i\left(\mu_l^2-\frac{1}{\eps^2}\right)\theta}
\widehat{(\overline{z^n}\psi^n)}_l(\theta)d\theta
+i\int_0^\tau\fe^{i\mu_l^2(\theta-\tau)}\widehat{(r^n\psi^n)}_l(\theta)d\theta.\nonumber
\end{align}
\end{subequations}
Combining the above results, we define local truncation error functions as
\be
\begin{split}
&\xi_\phi^n(x)=\sum_{l=-N/2}^{N/2-1}\widehat{(\xi^n_\phi)}_l \fe^{i\mu_l(x-a)},\qquad
\xi_\psi^n(x)=\sum_{l=-N/2}^{N/2-1}\widehat{(\xi^n_\psi)}_l  \fe^{i\mu_l(x-a)},\\
&\dot{\xi}_\phi^n(x)=\sum_{l=-N/2}^{N/2-1}\widehat{(\dot{\xi}^n_\phi)}_l \fe^{i\mu_l(x-a)},\qquad x\in \Omega,
\end{split}
\ee
where
\begin{subequations}\label{loc error}
\begin{align}
\widehat{(\xi^n_\phi)}_l:=&\widehat{\phi}_l(t_{n+1})-\Big[\cos(\omega_l\tau)\widehat{\phi}_l(t_n)+
\frac{\sin(\omega_l\tau)}{\omega_l}\widehat{\phi}_l'(t_n)+p_l(\tau)\widehat{(|\psi^n|^2)}_l(0)\nonumber\\
&+2q_l(\tau)\widehat{(\mathrm{Re}\{\overline{\psi^n}\varphi^n\})}_l(0)\Big],\label{loc error a}\\
\widehat{(\dot{\xi}^n_\phi)}_l:=&\widehat{\phi}_l'(t_{n+1})-\Big[-\omega_l\sin(\omega_l\tau)\widehat{\phi}_l(t_n)+
\cos(\omega_l\tau)\widehat{\phi}_l'(t_n)+p_l'(\tau)\widehat{(|\psi^n|^2)}_l(0)\nonumber\\
&+2q_l'(\tau)\widehat{(\mathrm{Re}\{\overline{\psi^n}\varphi^n\})}_l(0)\Big],\label{loc error b}\\
\widehat{(\xi^n_\psi)}_l:=&\widehat{\psi}_l(t_{n+1})-\fe^{-i\mu_l^2\tau}\widehat{\psi}_l(t_n)
-c_l^+(\tau)\widehat{(z^n\psi^n)}_l(0)-c_l^-(\tau)\widehat{(\overline{z^n}\psi^n)}_l(0)
\nonumber\\
&-d_l^+(\tau)\left[\widehat{(\mathfrak{z}^n\psi^n)}_l(0)+\widehat{(z^n\varphi^n)}_l(0)\right]
-d_l^-(\tau)\left[\widehat{(\overline{\mathfrak{z}^n}\psi^n)}_l(0)+\widehat{(\overline{z^n}\varphi^n)}_l(0)\right]
\nonumber\\
&-\frac{i\tau}{2}\widehat{(\gamma^{n+1}\psi^n)}_l(0)-\frac{i\tau^2}{2}\widehat{(\gamma^{n+1}\varphi^n)}_l.\label{loc error c}
\end{align}
\end{subequations}
Here we introduce three auxiliary functions as
\[\varphi^n(x)=\sum_{l\in\mathbb{Z}}\widehat{(\varphi^n)}_l\fe^{i\mu_l(x-a)},\ \mathfrak{z}^n(x)=\sum_{l\in\mathbb{Z}}\widehat{(\mathfrak{z}^n)}_l\fe^{i\mu_l(x-a)},\ \gamma^{n+1}(x)=\sum_{l\in\mathbb{Z}}\widehat{(\gamma^{n+1})}_l\fe^{i\mu_l(x-a)}, \]
with
\begin{numcases}
\,\widehat{(\varphi^n)}_l=\hspace{-0.5mm}-\frac{i}{\tau}\sin(\mu_l^2\tau)\widehat{(\psi^n)}_l(0)+i\widehat{(\phi^n\psi^n)}_l(0),\,
\widehat{(\mathfrak{z}^n)}_l=\hspace{-0.5mm}\frac{i}{2\tau}\sin(\mu_l^2\tau)\widehat{(z^n)}_l(0),\qquad\label{gamma est}\\
\widehat{(\gamma^{n+1})}_l=\hspace{-0.5mm}p_l(\tau)\widehat{(|\psi^n|^2)}_l(0)+2q_l(\tau)\widehat{(\mathrm{Re}\{\overline{\psi^n}\varphi^n\})}_l(0) -\frac{\sin(\omega_l \tau)}{\omega_l}\hspace{-0.5mm}\left(\widehat{(\mathfrak{z}^n)}_l
\hspace{-0.5mm}+\widehat{(\overline{\mathfrak{z}^n})}_l\right).\nonumber
\end{numcases}

Similarly to the case of the KG equation \cite{BZ} with details omitted here for brevity,
we can prove a prior estimates for the solution of the problem (\ref{MDF trun}) with (\ref{MDF ini}).

\begin{lemma}[A prior estimate of MDF] \label{lm: prior}
Let ${z}^n(x,s)$ and $r^n(x,s)$ be the solution of the MDF (\ref{MDF trun})
with initial boundary conditions (\ref{MDF ini}).
Under the assumption (\ref{assumption}), there exists a
constant $\tau_1>0$ independent of $\eps$, $h$ and $\tau$ (or $n$),
such that for $0<\tau\leq\tau_1$,
  \begin{align}
   &\left\|z^n\right\|_{L^\infty([0,\tau];H^{m_0+2})}+
   \left\|\partial_sz^n\right\|_{L^\infty([0,\tau];H^{m_0+1})}+
   \left\|\partial_{s}^2z_\pm^n\right\|_{L^\infty([0,\tau];H^{m_0})}\lesssim1,\label{reg z}\\
   &\left\|\partial_s^k r^n\right\|_{L^\infty([0,\tau];H^{6-k})}\lesssim
   {\eps^{2-2k}},\qquad k=0,1,2.\label{reg r}
  \end{align}
\end{lemma}

 Based on this result, when $0<\tau\le \tau_1$, it is easy to see that
 $\varphi^n,\mathfrak{z}^n\in H^{m_0}_p(\Omega),\gamma^{n+1}\in H^{4}_p(\Omega)$
and we have the prior estimates
$$\left\|\varphi^n\right\|_{H^4}\lesssim1,\quad \left\|\mathfrak{z}^n\right\|_{H^4}\lesssim1,\quad
\left\|\gamma^{n+1}\right\|_{H^4}\lesssim\eps^2.$$
With the above prior estimates, we can estimate the local truncation error functions.
\begin{lemma}[Estimates on local errors]\label{lm:local_error}
Under the assumption (\ref{assumption}), for any $0<\tau\leq\tau_1$ with $\tau_1$ given in Lemma \ref{lm: prior},
we have two independent bounds as
\begin{equation}
  \mcE\left(\xi_\phi^{n},\xi_\psi^n,\dot{\xi}_\phi^{n}\right)\lesssim \frac{\tau^6}{\eps^6},\quad\
  \mcE\left(\xi_\phi^{n},\xi_\psi^n,\dot{\xi}_\phi^{n}\right)\lesssim \tau^2\eps^2,
  \qquad 0\le n\le \frac{T}{\tau}-1.
  \label{lm:local_error_L result}
\end{equation}
\end{lemma}
\begin{proof}
Subtracting (\ref{loc error}) from (\ref{vcf: u cubic}) and by using the Taylor's expansion,
from (\ref{vcf: u cubic a}) and (\ref{loc error a}), we get
\be
\widehat{(\xi^n_\phi)}_l=\int_0^\tau\frac{\sin(\omega_l(\tau-\theta))}{\eps^2\omega_l}
\theta^2\Big[\int_0^1(1-s)\widehat{(|\psi^n|^2)}_l''(\theta s)ds\Big]d\theta
+2q_l(\tau)\widehat{(B_0)}_l,
\ee
where 
\[B_0(x)=\mathrm{Re}\{\overline{\psi^n(x,0)}(\partial_s\psi^n(x,0)-\varphi^n)\}, \qquad x\in\Omega.\]
Noticing (\ref{cldl765}), we obtain
\begin{align*}
|\widehat{(\xi^n_\phi)}_l|\lesssim&\frac{\tau^2}{\eps^2\omega_l}\int_0^\tau
\int_0^1\left|\widehat{(|\psi^n|^2)}_l''(\theta s)\right|ds d\theta+\frac{\tau^2}{\eps^2\omega_l}\left|\widehat{(B_0)}_l\right|.
\end{align*}
Noting $\frac{1}{\eps^2\omega_l}=\frac{1}{\sqrt{1+\eps^2\mu_l^2}}\le 1$ for $l=-\frac{N}{2},\ldots,\frac{N}{2}-1$
and  (\ref{assumption}), (\ref{abs psi}), we get
\begin{align}
\|\xi_\phi^n\|_{H^2}^2&\lesssim\tau^6\|\partial_{ss}\rho\|_{L^\infty([0,T];H^2)}^2
+\tau^4\left\|\psi^n(\cdot,0)(\partial_s\psi^n(\cdot,0)-\varphi^n)\right\|_{H^2}^2\nonumber\\
&\lesssim\tau^6+\tau^4\left\|\partial_s\psi^n(\cdot,0)-\varphi^n\right\|_{H^2}^2,
\qquad 0<\tau\leq\tau_1.\label{lm:local eq0}
\end{align}
By Parseval's identity, we have
\begin{align}
&\left\|\partial_s\psi^n(\cdot,0)-\varphi^n\right\|_{H^2}^2=
\sum_{l=-N/2}^{N/2-1}(1+\mu_l^2+\mu_l^4)
\left|\widehat{(\partial_s\psi^n)}_l(0)-\widehat{(\varphi^n)}_l\right|^2\nonumber\\
&\lesssim\sum_{l=-N/2}^{N/2-1}\mu_l^4\left|\frac{\sin(\mu_l^2\tau)}{\tau}
-\mu_l^2\right|^2|\widehat{(\psi^n)}_l(0)|^2
\lesssim\tau^2\sum_{l=-N/2}^{N/2-1}\mu_l^{12}|\widehat{(\psi^n)}_l(0)|^2\nonumber\\
&\lesssim\tau^2\|\psi\|_{L^\infty([0,T];H^6)}^2\lesssim\tau^2.\label{inter error}
\end{align}
This together with (\ref{lm:local eq0}), we get
\begin{align}
\|\xi_\phi^n\|_{H^2}^2\lesssim\tau^6,\quad 0<\tau\leq\tau_1.\label{lm:local eq1}
\end{align}
Similarly, noting  $\frac{|\mu_l|}{\eps^2\omega_l}=\frac{|\mu_l|}{\sqrt{1+\eps^2\mu_l^2}}\le\frac{1}{\eps}$
for $l=-\frac{N}{2},\ldots,\frac{N}{2}-1$, we obtain
\begin{align}
\|\partial_x\xi_\phi^n\|_{H^2}^2\lesssim\frac{\tau^6}{\eps^2} \quad \hbox{and}
\quad \|\dot{\xi}_\phi^n\|_{H^2}^2\lesssim\frac{\tau^6}{\eps^4}, \quad 0<\tau\leq\tau_1.
\label{lm:local eq2}
\end{align}
For $\xi_\psi^n$, from (\ref{vcf: u cubic c}) and (\ref{loc error c}) and by the Taylor's expansion, we have
\begin{align}\label{lm:local psi}
\widehat{(\xi^n_\psi)}_l=&i\fe^{-i\mu_l^2\tau}\int_0^\tau\fe^{i\left(\mu_l^2+\frac{1}{\eps^2}\right)\theta}\theta^2
\Big[\int_0^1(1-s)\widehat{(z^n\psi^n)}_l''(\theta s)ds\Big]d\theta \\
&+i\fe^{-i\mu_l^2\tau}\int_0^\tau\fe^{i\left(\mu_l^2-\frac{1}{\eps^2}\right)\theta}\theta^2
\Big[\int_0^1(1-s)\widehat{(\overline{z^n}\psi^n)}_l''(\theta s)ds\Big]d\theta\nonumber\\
&+d_l^+(\tau)\left[\widehat{(B_1)}_l+\widehat{(B_2)}_l\right]+d_l^-(\tau) \left[\widehat{(B_3)}_l+\widehat{(B_4)}_l\right]
+\frac{i\tau^3}{4}\widehat{(B_5)}_l
+\frac{i\tau}{2}\widehat{(B_6)}_l\nonumber\\
&+\frac{i\tau^2}{2}\left[\widehat{(B_7)}_l
+\widehat{(B_8)}_l\right]-\frac{i\tau}{2}\widehat{(r^n\psi^n)}_l(\tau)
+i\int_0^\tau\fe^{i\mu_l^2(\theta-\tau)}\widehat{(r^n\psi^n)}_l(\theta)d\theta,\nonumber
\end{align}
where
\beas
&&B_1(x)=\left[\partial_sz^n(x,0)-\mathfrak{z}^n(x)\right]\psi^n(x,0), \qquad
B_2(x)=z^n(x,0)\left[\partial_s\psi^n(x,0)
-\varphi^n(x\right],\\
&&B_3(x)=\left[\partial_s\overline{z^n}(x,0)-\overline{\mathfrak{z}^n}(x)\right]\psi^n(x,0),
\qquad B_4(x)=\overline{z^n}(x,0)\left[\partial_s\psi^n(x,0)-\varphi^n(x)\right],\\
&&B_5(x)=r^n(x,\tau)\partial_{ss}\psi^n(x,\theta_0),
\qquad B_6(x)=\left[r^n(x,\tau)-\gamma^{n+1}(x)\right]\psi^n(x,0), \\
&&B_7(x)=r^n(x,\tau)\left[\partial_s\psi^n(x,0)-\varphi^n(x)\right],\qquad
B_8(x)=\left[r^n(x,\tau)-\gamma^{n+1}(x)\right]\varphi^n(x),
\eeas
with $\theta_0\in[0,\tau]$. For the last two terms in (\ref{lm:local psi}),
we have
\begin{align*}
&\left|i\int_0^\tau\fe^{i\mu_l^2(\theta-\tau)}\widehat{(r^n\psi^n)}_l(\theta)d\theta
-\frac{i\tau}{2}\widehat{(r^n\psi^n)}_l(\tau)\right|\\
&\lesssim \int_0^\tau\theta(\tau-\theta)\left|\frac{d^2}{d\theta^2}\left[
\fe^{i\mu_l^2(\theta-\tau)}\widehat{(r^n\psi^n)}_l(\theta)\right]\right|d\theta.\nonumber
\end{align*}
Then by using the triangle inequality, we obtain
\begin{align*}
|\widehat{(\xi^n_\psi)}_l|\lesssim&\tau^2\int_0^\tau
\int_0^1\left[\left|\widehat{(z^n\psi^n)}_l''(\theta s)\right|+
\left|\widehat{(\overline{z^n}\psi^n)}_l''(\theta s)\right|\right]ds d\theta\nonumber\\
&+\tau^2\left[\left|\widehat{(B_1)}_l\right|+
\left|\widehat{(B_2)}_l\right|\right]
+\tau^2\left[\left|\widehat{(B_3)}_l\right|
+\left|\widehat{(B_4)}_l\right|\right]
+\tau^3\left|\widehat{(B_5)}_l\right|+
\tau\left|\widehat{(B_6)}_l\right|\nonumber\\
&+\tau^2\left|\widehat{(B_7)}_l\right|
+\tau^2\left|\widehat{(B_8)}_l\right|
+\tau^2\int_0^\tau\left|\frac{d^2}{d\theta^2}\left[
\fe^{i\mu_l^2(\theta-\tau)}\widehat{(r^n\psi^n)}_l(\theta)\right]\right|d\theta.
\end{align*}
Similarly as (\ref{inter error}), we find from (\ref{gamma est})
$$\|\partial_sz^n-\mathfrak{z}^n\|_{H^2}\lesssim\tau\|z^n(\cdot,0)\|_{H^6}
\lesssim\tau\left(\|\phi\|_{{L^\infty([0,T];H^6)}}
+\eps^2\|\partial_t\phi\|_{{L^\infty([0,T];H^6)}}\right)\lesssim\tau.$$
Again, similarly as (\ref{inter error}) and noticing (\ref{VCF r}),  we have
\begin{align*}
\|\gamma^{n+1}-r^n(\cdot,\tau)\|_{H^2}^2
&\lesssim\|\xi_\phi^n\|_{H^2}^2+\sum_{l\in\mathbb{Z}}(1+\mu_l^2+\mu_l^4)\left|
\frac{\sin(\omega_l\tau)}{\omega_l}\right|^2\cdot
\left|\widehat{(\mathfrak{z}^n)}_l-\widehat{(\partial_sz^n)}_l\right|^2\\
&\lesssim\tau^6+\tau^2\|\partial_sz^n-\mathfrak{z}^n\|_{H^2}^2\lesssim\tau^4.\nonumber
\end{align*}
By the assumption (\ref{assumption}) and Lemma \ref{lm: prior}, we get
\begin{align}
\left\|\xi^n_\psi\right\|_{H^2}^2&\lesssim\tau^6\left\|\partial_{ss}(z^n\psi^n)\right\|_{L^\infty([0,\tau];H^2)}^2
+\tau^6\|z^n\|_{L^\infty([0,\tau];H^6)}^2
+\tau^6\|\psi^n\|_{L^\infty([0,\tau];H^6)}^2\nonumber\\
&\quad+\tau^6\left\|r^n\partial_{ss}\psi^n\right\|_{L^\infty([0,\tau];H^2)}^2+
\tau^2\|\gamma^{n+1}-r^n(\cdot,\tau)\|_{H^2}^2
+\tau^6\left\|r^n\psi^n\right\|_{L^\infty([0,\tau];H^6)}^2\nonumber\\
&\quad+\tau^6\left\|\partial_s(r^n\psi^n)\right\|_{L^\infty([0,\tau];H^4)}^2
+\tau^6\left\|\partial_{ss}(r^n\psi^n)\right\|_{L^\infty([0,\tau];H^2)}^2
\lesssim\frac{\tau^6}{\eps^4}.\label{lm:local psi1}
\end{align}
Plugging (\ref{lm:local eq1}), (\ref{lm:local eq2}) and (\ref{lm:local psi1}) into (\ref{error engery}) with
$e_{\phi,N}^n=\xi_\phi^{n}$, $e_{\psi,N}^n=\xi_\psi^{n}$  and $\dot{e}_{\phi,N}^n=\dot{\xi}_\phi^{n}$, we immediately
get the first inequality in (\ref{lm:local_error_L result}).

Next we show how to get the second estimate in (\ref{lm:local_error_L result}). By using Taylor's expansion,
truncating at the first order term and  integrating by parts, we have
\begin{align*}
\widehat{(\xi^n_\phi)}_l=&\int_0^\tau\frac{\sin(\omega_l(\tau-\theta))}{\eps^2\omega_l}
\theta\Big[\int_0^1\widehat{(|\psi^n|^2)}_l'(\theta s)ds\Big]d\theta-2q_l(\tau)\widehat{(\mathrm{Re}\{\overline{\psi^n}\varphi^n\})}_l(0),\nonumber\\
=&q_l(\tau)\int_0^1\widehat{(|\psi^n|^2)}_l'(\tau s)ds
-\int_0^\tau q_l(\theta)\Big[\int_0^1\widehat{(|\psi^n|^2)}_l''(\theta s) s d s\Big]d\theta\\
&-q_l(0)\widehat{(|\psi^n|^2)}_l'(0)
-2q_l(\tau)\widehat{(\mathrm{Re}\{\overline{\psi^n}\varphi^n\})}_l(0).
\end{align*}
From  (\ref{cldl765}) and noticing $q_l(0)=0$ and $|q_l(\tau)|\lesssim\tau\eps^2$, we find
\begin{align}
\|\xi^n_\phi\|_{H^2}^2&\lesssim\tau^2\eps^4\left[\|\partial_t\rho\|_{L^\infty([0,T],H^2)}^2
+\|\partial_{tt}\rho\|_{L^\infty([0,T],H^2)}^2+
\|\psi\varphi^n\|_{L^\infty([0,T],H^2)}^2\right]\nonumber\\
&\lesssim\tau^2\eps^4,\qquad 0<\tau\le \tau_1, \qquad 0\le \eps\le 1.\label{lm:local xiphi1}
\end{align}
Similarly, we can get
\begin{align}
\|\partial_x\xi_\phi^n\|_{H^2}^2\lesssim\tau^2\eps^2 \quad \hbox{and}\quad
\|\dot{\xi}_\phi^n\|_{H^2}^2\lesssim\tau^2, \quad 0<\tau\leq\tau_1.
\label{lm:local xiphi2}
\end{align}
In order to estimate $\xi^n_\psi$, we introduce
\be\label{InJn987}
I^n(x)=\hspace{-2mm}\sum_{l=-N/2}^{N/2-1}\widehat{I}_l^n\fe^{i\mu_l(x-a)}, \qquad
J^n(x)=\hspace{-2mm}\sum_{l=-N/2}^{N/2-1}\widehat{J}_l^n\fe^{i\mu_l(x-a)},
\ee
where
\be
\begin{split}
&\widehat{I}_l^n=G_{1,l}^n-d_l^+(\tau)\left[\widehat{(\mathfrak{z}^n\psi^n)}_l(0)+\widehat{(z^n\varphi^n)}_l(0)\right],\\
&\widehat{J}^n_l=G_{2,l}^n-d_l^-(\tau)\left[\widehat{(\overline{\mathfrak{z}^n}\psi^n)}_l(0)+
\widehat{(\overline{z^n}\varphi^n)}_l(0)\right],
\end{split}
\ee
with
\be\label{G12ln}
\begin{split}
&G_{1,l}^n=i\fe^{-i\mu_l^2\tau}\int_0^\tau\fe^{i\left(\mu_l^2+\frac{1}{\eps^2}\right)\theta}\theta
\Big[\int_0^1\widehat{(z^n\psi^n)}_l'(\theta s)ds\Big]d\theta,\\
&G_{2,l}^n=i\fe^{-i\mu_l^2\tau}\int_0^\tau\fe^{i\left(\mu_l^2-\frac{1}{\eps^2}\right)\theta}\theta
\Big[\int_0^1\widehat{(\overline{z^n}\psi^n)}_l'(\theta s)ds\Big]d\theta.
\end{split}
\ee
Applying the Taylor's expansion in (\ref{vcf: u cubic c}) and (\ref{loc error c})
 and truncating at the first order term, we obtain
\be\label{lm:local xipsi2}
\widehat{(\xi^n_\psi)}_l=\widehat{I}_l^n+\widehat{J}_l^n
+i\int_0^\tau\fe^{i\mu_l^2(\theta-\tau)}\widehat{(r^n\psi^n)}_l(\theta)d\theta
-\frac{i\tau}{2}\left[\widehat{(\gamma^{n+1}\psi^n)}_l(0)+
\tau\widehat{(\gamma^{n+1}\varphi^n)}_l\right].
\ee
From (\ref{G12ln}), integration by parts, we get
\begin{align}
G_{1,l}^n&=i\fe^{-i\mu_l^2\tau}\int_0^\tau\fe^{i\left(\mu_l^2+\frac{1}{\eps^2}\right)\theta}\theta
\int_0^1\Big[\widehat{(\partial_sz^n\psi^n)}_l(\theta s)+\widehat{(z^n\partial_s\psi^n)}_l(\theta s)\Big]ds d\theta\nonumber\\
&=d_l^+(\tau)\int_0^1\widehat{(\partial_sz^n\psi^n)}_l(\tau s)d s-
\int_0^\tau d_l^+(\theta)\int_0^1\widehat{(\partial_{s}z^n\psi^n)}_l'(\theta s)s ds d\theta\nonumber\\
&\quad -\fe^{-i\mu_l^2\tau}\int_0^\tau\fe^{i\left(\mu_l^2+\frac{1}{\eps^2}\right)
\theta}\theta\int_0^1\Big[\widehat{(z^n\partial_{xx}\psi^n)}_l(\theta s)+
\widehat{(z^n\phi^n\psi^n)}_l(\theta s)\Big] ds d\theta\nonumber\\
&=d_l^+(\tau)\int_0^1\left[\widehat{(\partial_sz^n\psi^n)}_l(\tau s)+
i\widehat{(z^n\partial_{xx}\psi^n)}_l(\tau s)\right]ds\nonumber\\
&\quad-\int_0^\tau d_l^+(\theta)\int_0^1\left[\widehat{(\partial_{s}z^n\psi^n)}_l'(\theta s)+
i\widehat{(z^n\partial_{xx}\psi^n)}_l'(\theta s)\right]s ds d\theta-G_{3,l}^n,
\label{lm:local xipsi4}
\end{align}
where
\be
G_{3,l}^n=\fe^{-i\mu_l^2\tau}\int_0^\tau\fe^{i\left(\mu_l^2+\frac{1}{\eps^2}\right)\theta}\theta\int_0^1
\widehat{(z^n\phi^n\psi^n)}_l(\theta s) ds d\theta.
\ee
Noticing the ansatz (\ref{ansatz}), we get
\begin{align*}
G_{3,l}^n&=\fe^{-i\mu_l^2\tau}\int_0^1\Big[\int_0^\tau\fe^{i\left(\mu_l^2+\frac{1+s}{\eps^2}\right)\theta}\theta
\widehat{((z^n)^2\psi^n)}_l(\theta s)d\theta\\
&\quad+\int_0^\tau\fe^{i\left(\mu_l^2+\frac{1-s}{\eps^2}\right)\theta}\theta
\widehat{(|z^n|^2\psi^n)}_l(\theta s)d\theta+\int_0^\tau\fe^{i\left(\mu_l^2+\frac{1}{\eps^2}\right)\theta}\theta
\widehat{(z^nr^n\psi^n)}_l(\theta s)d\theta\Big] ds.
\end{align*}
Using integration by parts and the triangle inequality, we have
\begin{align}\label{G3ln}
\left|G_{3,l}^n\right|&\lesssim\int_0^1\Big[\tau\eps^2\left|\widehat{((z^n)^2\psi^n)}_l(\tau s)\right|+
\tau\eps^2\left|\widehat{(|z^n|^2\psi^n)}_l(\tau s)\right|+\eps^2\int_0^\tau\left|
\widehat{((z^n)^2\psi^n)}_l'(\theta s)\right|d\theta\nonumber\\
&\quad+\eps^2\int_0^\tau\left|
\widehat{(|z^n|^2\psi^n)}_l'(\theta s)\right|d\theta+\int_0^\tau\left|
\widehat{(z^nr^n\psi^n)}_l(\theta s)\right|d\theta\Big] ds.
\end{align}
Combining (\ref{InJn987})-(\ref{G3ln}) and noting $|d_l^+(\tau)|\lesssim\tau\eps^2$
and Lemma \ref{lm: prior}, we get
\begin{align}\label{In78954}
\|I^n\|_{H^2}\lesssim&\tau\eps^2\left[\|\partial_sz^n\psi^n\|_{L^\infty([0,\tau];H^2)}
+\|z^n\partial_{xx}\psi^n\|_{L^\infty([0,\tau];H^2)}+
\|\partial_{ss}z^n\psi^n\|_{L^\infty([0,\tau];H^2)}\right. \nonumber\\
&\left.+\|\partial_{s}z^n\partial_{s}\psi^n\|_{L^\infty([0,\tau];H^2)}
+\|\partial_{s}z^n\partial_{xx}\psi^n\|_{L^\infty([0,\tau];H^2)}
+\|z^n\partial_{sxx}\psi^n\|_{L^\infty([0,\tau];H^2)}\right.\nonumber\\
&\left.+\|(z^n)^2\psi^n\|_{L^\infty([0,\tau];H^2)}\right]
+\tau\eps^2\left[\|\mathfrak{z}^n\psi^n\|_{L^\infty([0,\tau];H^2)}+\|z^n
\varphi^n\|_{L^\infty([0,\tau];H^2)}\right]\nonumber\\
&+\tau\|z^nr^n\psi^n\|_{L^\infty([0,\tau];H^2)}\lesssim\tau\eps^2.
\end{align}
In order to estimate $J^n$, denote $N_\eps:=\frac{b-a}{2\sqrt{2}\pi\eps}$, when $|l|\le N_\eps$,
we have $|\mu_l|\leq\frac{1}{\sqrt{2}\eps}$ which indicates that $|d_l^-(\tau)|\lesssim\tau\eps^2$.
Similarly to the estimate of $I^n$, we have
\be\label{Jn6745}
\sum_{|l|\leq N_\eps}\left(1+\mu_l^2+\mu_l^4\right)|\widehat{J}^n_l|^2\lesssim\tau^2\eps^4.
\ee
On the other hand, when $|\mu_l|>\frac{1}{\sqrt{2}\eps}$, we have $|d_l^-(\tau)|\lesssim\tau^2$.
Due to the regularity of the solution,
i.e. the assumption (\ref{assumption}) and Lemma \ref{lm: prior}, we can get
\begin{align}\label{Jn6746}
&\sum_{|l|>N_\eps}\left(1+\mu_l^2+\mu_l^4\right)|\widehat{J}^n_l|^2\nonumber\\
&\lesssim\sum_{|l|>N_\eps}\tau^3\mu_l^4\left[
\int_0^\tau\int_0^1\left|\widehat{(\overline{z^n}\psi^n)}_l'(\theta s)\right|^2ds d\theta+\tau\left(\left|\widehat{(\overline{\mathfrak{z}^n}\psi^n)}_l(0)\right|^2
+\left|\widehat{(\overline{z^n}\varphi^n)}_l(0)\right|^2\right)\right]\nonumber\\
&\lesssim\tau^4\eps^4\left[\|\partial_s(z^n\psi^n)\|_{L^\infty([0,\tau];H^4)}^2+
\|\mathfrak{z}^n\psi^n(\cdot,0)\|_{H^4}^2
+\|z^n(\cdot,0)\varphi^n\|_{H^4}^2\right]\nonumber\\
&\lesssim\tau^4\eps^4\left[\|\partial_s(z^n\psi^n)\|_{L^\infty([0,\tau];H^4)}^2
+\frac{1}{\tau^2}\|z^n\psi^n\|_{L^\infty([0,\tau];H^4)}^2
+\|z^n\psi^n\phi^n\|_{L^\infty([0,\tau];H^4)}^2\right]\nonumber\\
&\lesssim\tau^2\eps^4.
\end{align}
From (\ref{Jn6745}) and (\ref{Jn6746}),  we obtain
\begin{align}
\|J^n\|_{H^2}\lesssim&\tau\eps^2, \qquad 0<\tau\le \tau_1, \qquad 0<\eps\le 1.
\end{align}
Combining the above estimates and noting (\ref{lm:local xipsi2}), we have
\begin{align}\label{lm:local xipsi5}
\|\xi_\psi^n\|_{H^2}^2\lesssim&\|I^n\|_{H^2}^2+\|J^n\|_{H^2}^2
+\tau^2\|r^n\psi^n\|_{L^\infty([0,\tau];H^2)}^2+\tau^4\|\gamma^{n+1}
\psi^n\|_{L^\infty([0,\tau];H^2)}^2\nonumber\\
&+\tau^4\|\gamma^{n+1}\varphi^n\|_{H^2}^2\lesssim\tau^2\eps^4.
\end{align}
Plugging (\ref{lm:local xiphi1}), (\ref{lm:local xiphi2}) and
(\ref{lm:local xipsi5}) into (\ref{error engery}), we get the second estimate.
\qed
\end{proof}

Defining the errors from the nonlinear terms as
\begin{equation}\left\{\begin{split}\label{eta def}
&\eta_\phi^n(x):=\sum_{l=-N/2}^{N/2-1}\widehat{(\eta_\phi^n)}_l\fe^{i\mu_l(x-a)},\quad \dot{\eta}_\phi^n(x):=\sum_{l=-N/2}^{N/2-1}\widehat{(\dot{\eta}_\phi^n)}_l\fe^{i\mu_l(x-a)},\\
&\eta_\psi^n(x):=\sum_{l=-N/2}^{N/2-1}\widehat{(\eta_\psi^n)}_l\fe^{i\mu_l(x-a)},\quad x\in\overline{\Omega},\quad n\ge0,
\end{split}\right.
\end{equation}
where
\begin{subequations}\label{eta l def}
\begin{align}
\widehat{(\eta_\phi^n)}_l\hspace{-0.5mm}=&p_l(\tau)\left[\widehat{(|\psi^n|^2)}_l(0)\hspace{-0.5mm}-\widetilde{(|\Psi^n|^2)}_l\right]\hspace{-0.5mm}
+2q_l(\tau)\left[\widehat{(\mathrm{Re}\{\overline{\psi^n}\varphi^n\})}_l(0)\hspace{-0.5mm}
-\widetilde{(\mathrm{Re}\{\overline{\Psi^n}\dot{\Psi}^n\})}_l\right],
\label{eta l def a}\\
\widehat{(\dot{\eta}_\phi^n)}_l\hspace{-0.5mm}=&p_l'(\tau)\left[\widehat{(|\psi^n|^2)}_l(0)\hspace{-0.5mm}
-\widetilde{(|\Psi^n|^2)}_l\right]\hspace{-0.5mm}+2
q_l'(\tau)\left[\widehat{(\mathrm{Re}\{\overline{\psi^n}\varphi^n\})}_l(0)\hspace{-0.5mm}
-\widetilde{(\mathrm{Re}\{\overline{\Psi^n}\dot{\Psi}^n\})}_l\right],
\label{eta l def b}\\
\widehat{(\eta_\psi^n)}_l=&c_l^+(\tau)\left[\widehat{(z^n\psi^n)}_l(0)-\widetilde{(Z^0\Psi^n)}_l\right]
+c_l^-(\tau)\left[\widehat{(\overline{z^n}\psi^n)}_l(0)-\widetilde{(\overline{Z^0}\Psi^n)}_l\right]\label{eta l def c}\\
&+d_l^+(\tau)\left[\widehat{(\mathfrak{z}^n\psi^n)}_l(0)-\widetilde{(\dot{Z}^0\Psi^n)}_l+
\widehat{(z^n\varphi^n)}_l(0)-\widetilde{(Z^0\dot{\Psi}^n)}_l\right]\nonumber\\
&+d_l^-(\tau)\left[\widehat{(\overline{\mathfrak{z}^n}\psi^n)}_l(0)-\widetilde{(\overline{\dot{Z}^0}\Psi^n)}_l+
\widehat{(\overline{z^n}\varphi^n)}_l(0)-\widetilde{(\overline{Z^0}\dot{\Psi}^n)}_l\right]\nonumber\\
&+\frac{i\tau}{2}\left[\widehat{(\gamma^{n+1}\psi^n)}_l(0)-\widetilde{(R^{n+1}\Psi^n)}_l
+\tau\left(\widehat{(\gamma^{n+1}\varphi^n)}_l-\widetilde{(R^{n+1}\dot{\Psi}^n)}_l\right)\right],\nonumber
\end{align}
\end{subequations}
then we have
\begin{lemma}[Estimates on $\eta_\phi^n$, $\eta_\psi^n$ and $\dot{\eta}_\phi^n$]\label{lm: nonlinear}
Under the assumption (\ref{assumption}) and assume (\ref{MTI sol bound}) holds (which will be proved by
induction later), we have for any $0<\tau\leq\tau_1$,
\begin{equation}\label{lm: nonlinear eq}
\mcE\left(\eta_\phi^{n},\eta_\psi^{n},\dot{\eta}_\phi^{n}\right)\lesssim \tau^2\mcE\left(e_{N,\phi}^{n},
e_{N,\psi}^{n},\dot{e}_{N,\phi}^{n}\right)+\frac{\tau^2 h^{2m_0}}{\eps^2},\quad 0\le
n\le \frac{T}{\tau}-1.
\end{equation}
\end{lemma}

\begin{proof}
From (\ref{eta def}) and (\ref{eta l def a}), we have
\begin{align}\label{etaphin78}
\|\eta_\phi^n\|_{H^2}&\lesssim\tau\||\psi^n(\cdot,0)|^2-I_N|\Psi^n|^2
\|_{H^2}+\tau^2\|\overline{\psi^n}(\cdot,0)\varphi^n
-I_N(\overline{\Psi^n}\dot{\Psi}^n)\|_{H^2}\nonumber\\
&\lesssim\tau\|I_N|\psi^n(\cdot,0)|^2-I_N|\Psi^n|^2\|_{H^2}+\tau^2\|I_N(\overline{\psi^n}(\cdot,0)\varphi^n)
-I_N(\overline{\Psi^n}\dot{\Psi}^n)\|_{H^2}+\tau h^{m_0}\nonumber\\
&\lesssim\tau\||\psi^n(\cdot,0)|^2-|\Psi^n|^2\|_{Y,2}+\tau^2\|\overline{\psi^n}(\cdot,0)\varphi^n
-\overline{\Psi^n}\dot{\Psi}^n\|_{Y,2}+\tau h^{m_0}\nonumber\\
&\lesssim\tau\|e_\psi^n\|_{Y,2}+\tau^2\|\varphi^n-\dot{\Psi}^n\|_{Y,2}+\tau h^{m_0}\nonumber\\
&\lesssim\tau\|e_\psi^n\|_{Y,2}+\tau^2\|\partial_{xx}^S\psi^n(\cdot,0)-\partial_{xx}^S\Psi^n\|_{Y,2}+
\tau^2\|\phi^n\psi^n(\cdot,0)-\Phi^n\Psi^n\|_{Y,2}+\tau h^{m_0}\nonumber\\
&\lesssim\tau\|e_\psi^n\|_{H^2}+\tau\|e_{\phi}^n\|_{H^2}+\tau^2\|\partial_{xx}^S\psi^n(\cdot,0)
-\partial_{xx}^S\Psi^n\|_{H^2}+\tau h^{m_0}\nonumber\\
&\lesssim\tau\|e_{\psi}^n\|_{H^2}+\tau\|e_{\phi}^n\|_{H^2}+\tau h^{m_0},
\end{align}
where the operator $\partial_{xx}^S$ on a function $f(x)$ is defined as
\begin{align*}
\partial_{xx}^Sf(x):
=-\sum_{l\in\mathbb{Z}}\frac{\sin(\mu_l^2\tau)}{\tau}\,\widehat{f}_l\;
\fe^{i\mu_l(x-a)}.
\end{align*}
Combining (\ref{etaphin78}) and (\ref{ephin896}), we get
\begin{align}
\|\eta_\phi^n\|_{H^2}
&\lesssim\tau\|e_{\psi,N}^n\|_{H^2}+\tau\|e_{\phi,N}^n\|_{H^2}+\tau h^{m_0}.\label{lm: nonlinear xiphi}
\end{align}
Similarly, we can obtain
\begin{align}\label{lm: nonlinear xidphi}
\eps^2\|\dot{\eta}_\phi^n\|_{H^2}^2+\|\partial_x\eta_\phi^n\|_{H^2}^2
&\lesssim\frac{\tau^2}{\eps^2}\left[\|e_{\psi,N}^n\|_{H^2}^2+\|e_{\phi,N}^n\|_{H^2}^2+h^{2m_0}\right].
\end{align}
Similarly, we can estimate all the terms of
$\eta^n_\psi$ in (\ref{eta l def c}) except the last term.
For the last term, we have
\begin{align}
&\sum_{l=-N/2}^{N/2-1}(1+\mu_l^2+\mu_l^4)\left|\widehat{(\gamma^{n+1}\psi^n)}_l(0)
-\widetilde{(R^{n+1}\Psi^n)}_l\right|^2\nonumber\\
&\lesssim\|I_N(\gamma^{n+1}\psi^n(\cdot,0))-I_N(R^{n+1}\psi^n))\|_{H^2}^2+h^{2m_0}\nonumber\\
&\lesssim\|e_\psi^n\|_{H^2}^2+\|\gamma^{n+1}-I_NR^{n+1}\|_{H^2}^2+h^{2m_0}.\label{lm: nonlinear eq1}
\end{align}
Noting (\ref{gamma est}) and (\ref{sine psu coeff r}), we get
\begin{align}
\left|\widehat{(\gamma^{n+1})}_l-\widetilde{(R^{n+1})}_l\right|
\lesssim&\left|\widehat{(z^n)}_l(0)-\widetilde{(Z^{0})}_l\right|+|p_l(\tau)|\cdot\left|\widehat{(|\psi^n|^2)}_l(0)
-\widetilde{(|\Psi^{n}|^2)}_l\right|\nonumber\\
&+|q_l(\tau)|\cdot\left|\widehat{(\overline{\psi^n}\varphi^n)}_l(0)-
\widetilde{(\overline{\Psi^{n}}\dot{\Psi}^{n})}_l\right|.
\end{align}
Combining the above two inequalities, we obtain
\begin{align}\label{lm: nonlinear eq2}
\|\gamma^{n+1}-I_NR^{n+1}\|_{H^2}\lesssim\|e_\phi^n\|_{H^2}+\eps^2
\|\dot{e}_\phi^n\|_{H^2}+\|e_\psi^n\|_{H^2}+h^{2m_0}.
\end{align}
Plugging (\ref{lm: nonlinear eq2}) into (\ref{lm: nonlinear eq1}), we get
\begin{align}\label{gm3456}
&\sum_{l=-N/2}^{N/2-1}(1+\mu_l^2+\mu_l^4)\left|\widehat{(\gamma^{n+1}\psi^n)}_l(0)
-\widetilde{(R^{n+1}\Psi^n)}_l\right|^2\nonumber\\
&\lesssim\|e_\psi^n\|_{H^2}^2+\|e_\phi^n\|_{H^2}^2+\eps^4\|\dot{e}_\phi^n\|_{H^2}^2+h^{2m_0}.
\end{align}
Similarly, we can obtain
\begin{align}\label{gm3458}
&\sum_{l=-N/2+1}^{N/2}(1+\mu_l^2+\mu_l^4)\left|\widehat{(\gamma^{n+1}\varphi^n)}_l
-\widetilde{(R^{n+1}\dot{\Psi}^n)}_l\right|^2\nonumber\\
&\lesssim\|e_\psi^n\|_{H^2}^2+\|e_\phi^n\|_{H^2}^2+\eps^4\|\dot{e}_\phi^n\|_{H^2}^2+h^{2m_0}.
\end{align}
Combining (\ref{gm3456}), (\ref{gm3458}) and (\ref{eta l def c}),  we have
\begin{align}\label{lm: nonlinear xipsi}
\|\eta_\psi^n\|_{H^2}&\lesssim\tau\left(\|e_\phi^n\|_{H^2}+
\eps^2\|\dot{e}_\phi^n\|_{H^2}+\|e_\psi^n\|_{H^2}+h^{m_0}\right)\nonumber\\
&\lesssim\tau\left(\|e_{\phi,N}^n\|_{H^2}+\eps^2\|\dot{e}_{\phi,N}^n\|_{H^2}+\|e_{\psi,N}^n\|_{H^2}+h^{m_0}\right).
\end{align}
Plugging (\ref{lm: nonlinear xiphi}), (\ref{lm: nonlinear xidphi}) and (\ref{lm: nonlinear xipsi})
into (\ref{error engery}), we get the estimate (\ref{lm: nonlinear eq}).
\qed
\end{proof}

\medskip

{\it Proof of Theorem \ref{main thm}.} The proof will be proceeded by the
energy method and the method of mathematical induction \cite{BDZ,Cai1,Zhao,Cai2,Dong}.
For $n=0$, from the initial data in the MTI-FP method (\ref{MTI-FP S})-(\ref{MTI-FP E})
 and noticing the assumption (\ref{assumption}), we have
\begin{align*}
&\|e_\phi^0\|_{H^2}+\|e_\psi^0\|_{H^2}+\eps^2\|\dot{e}_\phi^0\|_{H^2}\\
&=\|\phi_0-I_N\phi_0\|_{H^2}+\|\psi_0-I_N\psi_0\|_{H^2}+\|\phi_1-I_N\phi_1\|_{H^2}
\lesssim h^{m_0+2}\lesssim h^{m_0}.
\end{align*}
In addition, using the triangle inequality, we know that there exists $h_1>0$ independent of $\eps$ such that
for $0<h\leq h_1$
\begin{align*}
&\|\phi_I^0\|_{H^2}\leq\|\phi_0\|_{H^2}+\|e_\phi^0\|_{H^2}\leq C_\phi+1,\quad \|\dot{\phi}_I^0\|_{H^2}\leq\frac{\|\phi_1\|_{H^2}}{\eps^2}+\|\dot{e}_\phi^0\|_{H^2}\leq\frac{C_\phi+1}{\eps^2},\\
&\|\psi_I^0\|_{H^2}\leq\|\psi_0\|_{H^2}+\|e_\psi^0\|_{H^2}\leq C_\psi+1.
\end{align*}
Thus (\ref{MTI error bound1})-(\ref{MTI sol bound}) are valid for $n=0$.
Now  we assume that (\ref{MTI error bound1})-(\ref{MTI sol bound}) are valid for $0\leq n\leq m-1\leq T/\tau-1$.
Adding (\ref{loc error}) and (\ref{observ of scheme}), we have
\begin{subequations}\label{lm: error prog eq1}
\begin{align}
&\widehat{(e_\phi^{n+1})}_l=
\cos(\omega_l\tau)\widehat{(e_\phi^{n})}_l+
\frac{\sin(\omega_l\tau)}{\omega_l}\widehat{(\dot{e}_\phi^{n})}_l+\widehat{(\xi_\phi^n)}_l+\widehat{(\eta_\phi^n)}_l,\\
&\widehat{(\dot{e}_\phi^{n+1})}_l=
-\omega_l\sin(\omega_l\tau)\widehat{(e_\phi^{n})}_l+\cos(\omega_l\tau)\widehat{(\dot{e}_\phi^{n})}_l
+\widehat{(\dot{\xi}_\phi^n)}_l+\widehat{(\dot{\eta}_\phi^n)}_l,\\
&\widehat{(e_\psi^{n+1})}_l=\fe^{-i\mu_l^2\tau}\widehat{(e_\psi^{n})}_l+\widehat{(\xi_\psi^n)}_l+\widehat{(\eta_\psi^n)}_l.
\end{align}
\end{subequations}
Using the Cauchy's inequality, we obtain
\begin{subequations}
\begin{align}
&\left|\widehat{(e_\phi^{n+1})}_l\right|^2\leq(1+\tau)\left|\cos(\omega_l\tau)\widehat{(e_\phi^{n})}_l
+\frac{\sin(\omega_l\tau)}{\omega_l}\widehat{(\dot{e}_\phi^{n})}_l\right|^2+\frac{1+\tau}
{\tau}\left|\widehat{(\xi_\phi^n)}_l+\widehat{(\eta_\phi^n)}_l\right|^2,
\label{lm:error prog eq2}\\
&\left|\widehat{(\dot{e}_\phi^{n+1})}_l\right|^2\leq(1+\tau)\left|\cos(\omega_l\tau)\widehat{(\dot{e}_\phi^{n})}_l
-\omega_l\sin(\omega_l\tau)\widehat{(e_\phi^{n})}_l\right|^2
+\frac{1+\tau}{\tau}\left|\widehat{(\dot{\xi}_\phi^n)}_l+\widehat{(\dot{\eta}_\phi^n)}_l\right|^2,\label{lm:error prog eq3}\\
&\left|\widehat{(e_\psi^{n+1})}_l\right|^2\leq(1+\tau)\left|\widehat{(e_\psi^{n})}_l\right|^2+\frac{1+\tau}
{\tau}\left|\widehat{(\xi_\psi^n)}_l+\widehat{(\eta_\psi^n)}_l\right|^2.\label{lm:error prog eqpsi}
\end{align}
\end{subequations}
Multiplying (\ref{lm:error prog eq2})-(\ref{lm:error prog eqpsi}) by
 $(\mu_l^2+\frac{1}{\eps^2})(1+\mu_l^2+\mu_l^4)$,
$\eps^2(1+\mu_l^2+\mu_l^4)$ and $\frac{1}{\eps^2}(1+\mu_l^2+\mu_l^4)$, respectively, then adding them together
and summing up for $l=-N/2,\\ \ldots, N/2-1$,
we obtain
\[\mathcal{E}(e^{n+1}_{\phi,N},e^{n+1}_{\psi,N},\dot{e}^{n+1}_{\phi,N})\leq (1+\tau)\mathcal{E}(e^{n}_{\phi,N},e^{n}_{\psi,N},\dot{e}^{n}_N)+\frac{1+\tau}{\tau}
\mathcal{E}(\xi_\phi^n+\eta_\phi^n,\xi_\psi^n+\eta_\psi^n,\dot{\xi}_\phi^n+\dot{\eta}_\phi^n).\]
Using the Cauchy's inequality, we get
\begin{align}
&\mathcal{E}(e^{n+1}_{\phi,N},e^{n+1}_{\psi,N},\dot{e}^{n+1}_{\phi,N})
-\mathcal{E}(e^{n}_{\phi,N},e^{n}_{\psi,N},\dot{e}^{n}_{\phi,N})\nonumber\\
&\lesssim\tau\mathcal{E}(e^{n}_{\phi,N},e^{n}_{\psi,N},\dot{e}^{n}_{\phi,N})
+\frac{1+\tau}{\tau}\left[\mathcal{E}(\xi_\phi^n,\xi_\psi^n,\dot{\xi}_\phi^n)
+\mathcal{E}(\eta_\phi^n,\eta_\psi^n,\dot{\eta}_\phi^n)\right].\label{proof: eq1}
\end{align}
Inserting  (\ref{lm: nonlinear eq}) and the second inequality in (\ref{lm:local_error_L result})
 into (\ref{proof: eq1}), we get
\bea\small
\mathcal{E}(e^{n+1}_{\phi,N},e^{n+1}_{\psi,N},\dot{e}^{n+1}_{\phi,N})\hspace{-0.5mm}
-\hspace{-0.5mm}\mathcal{E}(e^{n}_{\phi,N},e^{n}_{\psi,N},\dot{e}^{n}_{\phi,N})
\hspace{-0.5mm}\lesssim\hspace{-0.5mm}\tau\mathcal{E}(e^{n}_{\phi,N},e^{n}_{\psi,N},\dot{e}^{n}_{\phi,N})
+\frac{\tau h^{2m_0}}{\eps^2}+\frac{\tau^5}{\eps^6}.\qquad\quad
\eea
Summing the above inequality for $0\le n\le m-1$ and then applying the discrete Gronwall's inequality, we have
\begin{equation}\label{emn865}
\mathcal{E}(e^{n}_{\phi,N},e^{n}_{\psi,N},\dot{e}^{n}_{\phi,N})\lesssim\frac{h^{2m_0}}{\eps^2}+\frac{\tau^4}{\eps^6}.
\end{equation}
Similarly, by using the first inequality in  (\ref{lm:local_error_L result}), we obtain
\begin{equation}\label{emn835}
\mathcal{E}(e^{n}_{\phi,N},e^{n}_{\psi,N},\dot{e}^{n}_{\phi,N})\lesssim \frac{h^{2m_0}}{\eps^2}+\eps^2.
\end{equation}
Combining (\ref{error engery}), (\ref{emn865}) and (\ref{emn835}),
we get that (\ref{MTI error bound1}) is valid for $n=m$, which implies
\[\|e_\phi^m\|_{H^2}+\|e_\psi^m\|_{H^2}+\eps^2\|\dot{e}_\phi^m\|_{H^2}\leq h^{m_0}+\tau.\]
Using the triangle inequality, we obtain that these exist $h_2>0$ and $\tau_2>0$ independent of
$\eps$ such that
\[\begin{split}
&\|\phi_I^m\|_{H^2}\leq \|\phi(\cdot,t_m)\|_{H^2}+\|e_\phi^m\|_{H^2}\leq C_\phi+1,\\
&\|\psi_I^m\|_{H^2}\leq \|\psi(\cdot,t_m)\|_{H^2}+\|e_\psi^m\|_{H^2}\leq C_\psi+1,
\qquad 0<h\le h_2,\quad  0<\tau\le \tau_2,\\
&\|\dot{\phi}_I^m\|_{H^2}\leq \|\partial_t\phi(\cdot,t_m)\|_{H^2}+\|\dot{e}_\phi^m\|_{H^2}
\leq \frac{C_\phi+1}{\eps^2}.
\end{split}
\]
Thus (\ref{MTI sol bound}) is also valid for $n=m$. Then the proof is completed
by choosing $\tau_0=\min\{\tau_1,\tau_2\}$ and $h_0=\min\{h_1,h_2\}$.
\qed

%

\begin{remark}
Here we emphasize that Theorem \ref{main thm} holds in two/three
dimensions (2D/3D) and the above approach can be directly
extended to 2D/3D
without extra effort.
The only thing needs to be taken care of is the Sobolev inequality used
in Lemma \ref{lm: nonlinear} in 2D/3D
\begin{equation}
\|u\|_{L^\infty(\Omega)}\lesssim\|u\|_{H^2(\Omega)},\qquad
\|u\|_{W^{1,p}(\Omega)}\lesssim \|u\|_{H^2(\Omega)},\quad 1<p<6,
\end{equation}
where $\Omega$ is a bounded domain in 2D/3D.  By using the assumption
\eqref{MTI sol bound}, Lemma \ref{lm: nonlinear} still holds in 2D/3D.
Thus \eqref{MTI sol bound} and the error bounds can be proved by
the method of mathematical induction since our scheme
is explicit.
\end{remark}

\begin{remark}
Under a weaker assumption on the regularity of the solution, i.e.
there exists an integer $m_0\ge 4$
such that $\phi\in C^1([0,T];H_p^{m_0+3}(\Omega))$, $\psi\in C([0,T]; H_p^{m_0+1}(\Omega))\\ \cap C^1([0,T];H_p^{m_0-1}(\Omega))
\cap C^2([0,T];H_p^{m_0-3}(\Omega))$ and
\be\label{assumption23}
\begin{split}
&\left\|\phi\right\|_{L^\infty([0,T]; H^{m_0+3})}+\eps^2 \left\|\partial_t\phi\right\|_{L^\infty([0,T]; H^{m_0+3})}\lesssim 1,\\
&\left\|\psi\right\|_{L^\infty([0,T]; H^{m_0+1})}+\left\|\partial_t\psi\right\|_{L^\infty([0,T]; H^{m_0-1})}+
\eps^2\left\|\partial_{tt}\psi\right\|_{L^\infty([0,T]; H^{m_0-3})}\lesssim 1,
\end{split}
\ee
we can establish an $H^1$-norm estimate of the MTI-FP method in
Theorem \ref{main thm} by a very similar proof with all the $H^2$-norm in the proof being changed into $H^1$-norm.
In 1D case, the $H^1$-norm estimate of the MTI-FP method holds without any stability (or CFL) condition.
However in 2D/3D, due to the use of the inverse inequality to provide the bound in $l^\infty$-norm
of the numerical solution \cite{Cai1,BZ,Cai2},
we need impose the technical condition
\begin{equation}\label{stab987}
\tau\lesssim \rho_d(h),\qquad \hbox{with} \quad \rho_d(h)=
\left\{\begin{array}{ll}   1/|\ln h|,\quad &d=2,\\ \sqrt{h},\quad &d=3.\\
\end{array}\right.
\end{equation}
Of course, if the solution of the KGS is smooth enough, we can always
adapt the $H^2$-norm estimate in Theorem \ref{main thm} and thus there is
no need to assume the stability
(or CFL) condition (\ref{stab987}).
\end{remark}

\section{Numerical results}\label{sec: result}
In this section, we report numerical results
to demonstrate the uniform convergence of the MTI-FP method for
$\eps\in(0,1]$ and apply it to numerically study convergence rates of the KGS equations to its limiting
models (\ref{modellimit92}) and (\ref{modellimit4}) in the nonrelativistic limit regime. In order to do so,
we take $d=1$ and $\mu=\lambda=1$ in (\ref{KGS}) and choose the initial data in (\ref{KGS ini})
with $d=1$ as
\begin{align}\label{numer example}
\psi_0(x)=\frac{1+i}{2}\sech\left(\frac{x^2}{2}\right),\quad \phi_0(x)=\frac{1}{2}\fe^{-x^2},\quad
\phi_1(x)=\frac{1}{\sqrt{2}}\fe^{-x^2},\quad x\in{\mathbb R}.\end{align}

\subsection{Accuracy test}
The problem (\ref{KGS}) with $d=1$ and (\ref{numer example})
is solved on a bounded interval $\Omega=[-32,32]$, i.e. $b=-a=32$,
which is large enough to guarantee that the periodic boundary condition does not
 introduce a significant aliasing error relative to the original problem.
To quantify the error, we introduce two error functions:
\begin{align*}
&\fe^{\tau,h}_{\phi,\eps}(t=t_n):=\left\|\phi(\cdot,t_n)-\phi^n_I\right\|_{H^2},\quad
\fe^{\tau,h}_{\phi,\infty}(t):=\max_{0<\eps\le1}\left\{\fe^{\tau,h}_{\phi,\eps}(t)\right\},\\
&\fe^{\tau,h}_{\psi,\eps}(t=t_n):=\left\|\psi(\cdot,t_n)-\psi^n_I\right\|_{H^2},\quad
\fe^{\tau,h}_{\psi,\infty}(t):=\max_{0<\eps\le1}\left\{\fe^{\tau,h}_{\psi,\eps}(t)\right\}.
\end{align*}
Since the analytical solution to this problem is not available,
so the `reference' solution here is obtained numerically by the MTI-FP method
(\ref{MTI-FP S})-(\ref{MTI-FP E}) with very fine mesh $h=1/32$ and
time step $\tau=5\times 10^{-6}$.
Tab. \ref{tb: MTIFP spa psi} and Tab. \ref{tb: MTIFP spa phi}
show the spatial error of the MTI-FP method for $\psi$
and $\phi$, respectively, at $t=1$
under different $\eps$ and $h$ with a very small time step $\tau=5\times 10^{-6}$
such that the discretization error in time is negligible.
Tab. \ref{tb: MTIFP tem psi} and Tab. \ref{tb: MTIFP tem phi} show the
temporal error of the MTI-FP method for $\psi$
and $\phi$, respectively, at $t=1$ under different $\eps$ and $\tau$
with a small mesh size $h=1/16$ such that
the discretization error in space is negligible.

\begin{table}[t!]
\tabcolsep 0pt
\caption{Spatial error analysis: $\fe^{\tau,h}_{\psi,\eps}(t=1)$ with  $\tau=5\times 10^{-6}$
for different $\eps$ and $h$.}\label{tb: MTIFP spa psi}
\begin{center}
\begin{tabular*}{1\textwidth}{@{\extracolsep{\fill}}llllll}
\hline
$\fe^{\tau,h}_{\psi,\eps}(t=1)$         & $h_0=1$	      &$h_0/2$	     
   &$h_0/4$	        &$h_0/8$	&$h_0/16$\\
\hline
$\eps_0=1$	                   &6.90E-1	&1.39E-1	&1.70E-3	&2.44E-7       &8.96E-11\\
$\eps_0/2$	                   &7.57E-1	&1.40E-1	&1.70E-3	&2.40E-7       &8.37E-11\\
$\eps_0/2^2$	               &6.37E-1	&1.45E-1	&1.70E-3	&2.40E-7       &9.19E-11\\
$\eps_0/2^3$	               &6.46E-1	&1.30E-1	&1.90E-3	&2.40E-7       &8.89E-11\\
$\eps_0/2^4$	               &6.47E-1	&1.31E-1	&1.80E-3	&2.38E-7       &9.26E-11\\
$\eps_0/2^5$	               &6.47E-1	&1.31E-1	&1.80E-3	&2.38E-7       &9.66E-11\\
$\eps_0/2^6$	               &6.46E-1	&1.27E-1	&1.70E-3	&2.52E-7       &8.76E-11\\
$\eps_0/2^7$	               &6.46E-1	&1.31E-1	&1.80E-3	&2.38E-7       &9.08E-11\\
$\eps_0/2^8$	               &6.46E-1	&1.27E-1	&1.70E-3	&2.52E-7       &9.08E-11\\
$\eps_0/2^{9}$	               &6.46E-1	&1.31E-1	&1.80E-3	&2.38E-7       &1.02E-10\\
$\eps_0/2^{11}$	               &6.46E-1	&1.31E-1	&1.80E-3	&2.38E-7       &9.07E-11\\
$\eps_0/2^{13}$	               &6.46E-1	&1.31E-1	&1.80E-3	&2.38E-7       &8.61E-11\\
\hline
\end{tabular*}
\end{center}
\end{table}

\begin{table}[h!]
\tabcolsep 0pt
\caption{Spatial error analysis: $\fe^{\tau,h}_{\phi,\eps}(t=1)$ with  $\tau=5\times 10^{-6}$
for different $\eps$ and $h$.}\label{tb: MTIFP spa phi}
\begin{center}
\begin{tabular*}{1\textwidth}{@{\extracolsep{\fill}}llllll}
\hline
$\fe^{\tau,h}_{\phi,\eps}(t=1)$         & $h_0=1$	      &$h_0/2$	        &$h_0/4$	        &$h_0/8$	&$h_0/16$\\
\hline
$\eps_0=1$	                   &4.74E-2	&1.60E-3	&3.64E-6	&3.57E-10       &9.23E-11\\
$\eps_0/2$	                   &8.29E-2	&1.90E-3	&1.23E-5	&3.92E-10       &6.18E-11\\
$\eps_0/2^2$	               &2.32E-1	&1.12E-2	&4.61E-5	&1.46E-\,9      &5.24E-11\\
$\eps_0/2^3$	               &3.49E-1	&1.30E-3	&9.31E-5	&5.33E-10       &4.58E-11\\
$\eps_0/2^4$	               &2.64E-1	&1.20E-3	&3.35E-5	&5.16E-\,9      &4.49E-11\\
$\eps_0/2^5$	               &2.71E-1	&2.09E-4	&6.92E-6	&8.76E-10       &5.49E-11\\
$\eps_0/2^6$	               &1.13E-1	&5.27E-4	&2.56E-6	&1.73E-10       &6.04E-11\\
$\eps_0/2^7$	               &3.99E-1	&4.29E-4	&6.28E-7	&4.23E-11       &6.20E-11\\
$\eps_0/2^8$	               &6.74E-2	&5.90E-4	&7.20E-8	&2.38E-11       &6.20E-11\\
$\eps_0/2^{9}$	               &3.37E-1	&5.13E-5	&2.78E-8	&1.86E-11       &6.51E-11\\
$\eps_0/2^{11}$	               &3.09E-1	&7.75E-4	&2.98E-9	&1.73E-11       &6.19E-11\\
$\eps_0/2^{13}$	               &1.87E-1	&3.64E-4	&3.33E-9	&1.72E-11       &6.13E-11\\
\hline
\end{tabular*}
\end{center}
\end{table}

\begin{table}[t!]
\tabcolsep 0pt
\caption{Temporal error analysis: $\fe^{\tau,h}_{\psi,\eps}(t=1)$ and
$\fe^{\tau,h}_{\psi,\infty}(t=1)$ with $h=1/16$ for different $\eps$ and $\tau$.}\label{tb: MTIFP tem psi}
\begin{center}
\begin{tabular*}{1\textwidth}{@{\extracolsep{\fill}}llllllll}
\hline
$\fe^{\tau,h}_{\psi,\eps}(t=1)$      & $\tau_0=0.2$	&$\tau_0/2^2$	
&$\tau_0/2^4$	&$\tau_0/2^6$	 &$\tau_0/2^8$  & $\tau_0/2^{10}$  &$\tau_0/2^{12}$\\
\hline
$\eps_0=1$	            &1.31E-1	&1.60E-2	&1.00E-3	&6.20E-5	&3.84E-6	&2.35E-7	&1.77E-8\\
rate                    &---	    &1.52	    &2.00	    &2.01	    &2.01	    &2.02	    &1.87\\ \hline
$\eps_0/2$	        &1.59E-1	&2.05E-2	&1.30E-3	&7.70E-5	&4.76E-6	&2.92E-7	&1.37E-8\\
rate                    &---	    &1.48	    &1.99	    &2.04	    &2.01	    &2.01	    &2.20\\ \hline
$\eps_0/2^2$	        &1.20E-1	&1.48E-2	&1.20E-3	&7.65E-5	&4.79E-6	&2.95E-7	&1.38E-8\\
rate                    &---	    &1.35	    &1.81	    &1.99	    &2.00	    &2.01	    &2.20\\ \hline
$\eps_0/2^3$	        &2.23E-2	&5.30E-3	&5.19E-4	&3.20E-5	&2.00E-6	&1.23E-7	&5.77E-9\\
rate                    &---	    &1.04	    &1.68	    &2.00	    &2.00	    &2.01	    &2.20\\ \hline
$\eps_0/2^4$	        &9.50E-3	&3.50E-3	&2.20E-3	&1.32E-4	&8.06E-6	&4.94E-7	&2.31E-8\\
rate                    &---	    &0.72	    &0.33	    &2.03	    &2.02	    &2.01	    &2.20\\ \hline
$\eps_0/2^5$	        &1.40E-3	&1.00E-3	&6.34E-4	&4.67E-4	&3.35E-5	&2.02E-6	&9.41E-8\\
rate                    &---	    &0.24	    &0.33	    &0.22	    &1.91	    &2.02	    &2.20\\ \hline
$\eps_0/2^6$	        &3.82E-4	&5.65E-5	&1.43E-4	&1.51E-4	&1.16E-4	&8.29E-6	&3.80E-7\\
rate                    &---	    &1.38	    &-0.67	    &-0.18   	&0.19	    &1.90	    &2.22\\ \hline
$\eps_0/2^7$	        &5.64E-5	&2.30E-5	&1.28E-5	&3.61E-5	&3.83E-5	&2.89E-5	&1.58E-6\\
rate                    &---	    &0.65	    &0.42	    &-0.75	    &-0.04	    &0.20	    &2.10\\ \hline
$\eps_0/2^8$	        &2.10E-5	&5.60E-6	&6.41E-6	&2.61E-6	&1.10E-5	&1.12E-5	&5.70E-6\\
rate                    &---	    &0.95	    &-0.10	    &0.65	    &-1.00	    &-0.01	    &0.49\\ \hline
$\eps_0/2^{9}$	        &5.82E-6	&4.05E-6	&2.40E-6	&2.00E-6	&2.41E-6	&8.49E-7	&2.50E-7\\
rate                    &---	    &0.26	    &0.38	    &0.13	    &-0.13	    &0.75	    &0.88\\ \hline
$\eps_0/2^{11}$	        &2.09E-7	&2.19E-7	&1.67E-7	&1.58E-7	&8.48E-8	&1.43E-7	&4.63E-8\\
rate                    &---	    &-0.04	    &0.20	    &0.04	    &0.45	    &-0.38	    &0.81\\ \hline
$\eps_0/2^{13}$	        &1.78E-8	&8.93E-9	&5.82E-9	&5.68E-9	&7.48E-9	&1.92E-9	&9.46E-9\\
rate                    &---	    &0.50	    &0.31	    &0.02	    &-0.20	    &0.98	    &-1.15\\ \hline \hline
$\fe^{\tau,h}_{\psi,\infty}(t=1)$  &1.59E-1	&2.05E-2	&2.20E-3	&4.67E-4	&1.16E-4	&2.89E-5	&5.70E-6\\
rate                    &---	 &1.48	    &1.61	    &1.11	    &1.00	    &1.00	    &1.17\\
\hline
\end{tabular*}
\end{center}
\end{table}

\begin{table}[t!]
\tabcolsep 0pt
\caption{Temporal error analysis: $\fe^{\tau,h}_{\phi,\eps}(t=1)$ and
$\fe^{\tau,h}_{\phi,\infty}(t=1)$ with $h=1/16$ for different $\eps$ and $\tau$.}\label{tb: MTIFP tem phi}
\begin{center}
\begin{tabular*}{1\textwidth}{@{\extracolsep{\fill}}llllllll}
\hline
$\fe^{\tau,h}_{\phi,\eps}(t=1)$      & $\tau_0=0.2$	&$\tau_0/2^2$	&$\tau_0/2^4$	&$\tau_0/2^6$	 &$\tau_0/2^8$  & $\tau_0/2^{10}$  &$\tau_0/2^{12}$\\
\hline
$\eps_0=1$	            &3.74E-2	&1.50E-3	&3.82E-5	&1.78E-6	&1.09E-7	&6.73E-9	&3.32E-10\\
rate                    &---	    &2.32	    &2.65	    &2.21	    &2.01	    &2.01       &2.17\\ \hline
$\eps_0/2$	        &9.32E-2	&4.80E-3	&2.29E-4	&1.35E-5	&8.32E-7	&5.11E-8	&2.39E-\,9\\
rate                    &---	    &2.14	    &2.19	    &2.04	    &2.01	    &2.01	    &2.20\\ \hline
$\eps_0/2^2$	        &9.65E-2	&1.26E-2	&7.09E-4	&4.23E-5	&2.61E-6	&1.60E-7	&7.51E-\,9\\
rate                    &---	    &1.47	    &2.07	    &2.03	    &2.01	    &2.01	    &2.20\\ \hline
$\eps_0/2^3$	        &2.33E-2	&4.10E-3	&3.02E-4	&1.95E-5	&1.21E-6	&7.47E-8	&3.50E-\,9\\
rate                    &---	    &1.25	    &1.88	    &1.98	    &2.00	    &2.01	    &2.20\\ \hline
$\eps_0/2^4$	        &8.70E-3	&3.30E-3	&9.09E-5	&6.00E-6	&3.82E-7	&2.36E-8	&1.11E-\,9\\
rate                    &---	    &0.70	    &2.59	    &1.96	    &1.99	    &2.00	    &2.20\\ \hline
$\eps_0/2^5$	        &9.59E-4	&8.29E-4	&2.42E-4	&3.73E-6	&2.89E-7	&1.81E-8	&8.53E-10\\
rate                    &---	    &0.11	    &0.89	    &3.01	    &1.85	    &2.00       &2.20\\ \hline
$\eps_0/2^6$	        &4.71E-4	&1.19E-4	&1.33E-5	&3.80E-6	&2.13E-7	&1.67E-8	&8.01E-10\\
rate                    &---	    &0.99	    &1.58	    &0.90	    &2.07	    &1.84	    &2.19\\ \hline
$\eps_0/2^7$	        &8.27E-5	&2.88E-5	&2.49E-6	&1.76E-7	&4.95E-8	&1.30E-8	&7.86E-10\\
rate                    &---	    &0.76	    &1.76	    &1.91	    &0.92	    &0.96	    &2.02\\ \hline
$\eps_0/2^8$	        &3.33E-5	&5.69E-6	&3.28E-7	&2.18E-8	&3.37E-9	&1.70E-9	&6.57E-10\\
rate                    &---	    &1.28	    &2.06	    &1.95	    &1.35	    &0.49	    &0.69\\ \hline
$\eps_0/2^{9}$	        &8.92E-6	&3.98E-6	&1.04E-7	&5.46E-9	&3.66E-10	&5.75E-11	&2.04E-11\\
rate                    &---	    &0.58	    &2.63	    &2.12	    &1.95	    &1.34	    &0.75\\ \hline
$\eps_0/2^{11}$	        &3.05E-7	&1.59E-7	&5.64E-9	&1.32E-9	&7.21E-11	&7.09E-11	&7.43E-11\\
rate                    &---	    &0.47	    &2.41	    &1.05	    &2.10	    &0.01	    &-0.03\\ \hline
$\eps_0/2^{13}$	        &2.30E-8	&7.26E-9	&2.21E-9	&7.77E-10	&7.74E-10	&7.74E-10	&7.74E-10\\
rate                    &---	    &0.83	    &0.86	    &0.76	    &0.01	    &0.00	    &0.00\\ \hline \hline
$\fe^{\tau,h}_{\phi,\infty}(t=1)$  &9.65E-2	&1.26E-2	&7.09E-4	&4.23E-5	&2.61E-6	&1.60E-7	&7.51E-9\\
rate                    &---	 &1.46	    &2.00	    &2.00	    &2.00	   &2.00	   &2.00\\
\hline
\end{tabular*}
\end{center}
\end{table}

From Tabs. \ref{tb: MTIFP spa psi}-\ref{tb: MTIFP tem phi} and extensive additional results not shown here for brevity,
we can draw the following observations:

(i) The MTI-FP method is spectrally accurate in space, which is uniformly for $0<\eps\le1$
(cf. Tabs. \ref{tb: MTIFP spa psi}\&\ref{tb: MTIFP spa phi}). The approximation in $\phi$
is more accurate than $\psi$.

(ii) The MTI-FP method converges uniformly in time for $\eps\in(0,1]$
with linear/quadratical convergence rate for $\psi$/$\phi$, respectively
 (cf. last row in Tabs. \ref{tb: MTIFP tem psi}\&\ref{tb: MTIFP tem phi}).
In addition, for each fixed $\eps=\eps_0>0$, when $\tau$ is small enough,
it converges quadratically in time (cf. each row in the upper triangle of Tabs.
\ref{tb: MTIFP tem psi}\&\ref{tb: MTIFP tem phi});
for fixed $\tau$, when $0<\eps\ll1$, it converges
quadratically in terms of $\eps$ (cf. each column in the lower
triangle of Tabs. \ref{tb: MTIFP tem psi}\&\ref{tb: MTIFP tem phi}).

(iii) The practical temporal uniform convergence order of the MTI-FP in $\phi$ is better than the theoretical
 error estimate result. In fact, the temporal local truncation error of the MTI-FP in $\phi$ is at $O(\tau^3)$,
but due to the nonlinear coupling with $\psi$, the theoretical error bound drops down to uniform linear convergence rate.
The rigorous mathematical analysis is still on-going.

(iv) In the nonrelativistic limit regime, i.e. $0<\eps\ll1$, the mesh strategy
(or $\eps$-scalability) of the MTI-FP method is  $\tau=O(1)$ and $h=O(1)$.

\subsection{Convergence rates of KGS to its limiting models}
Let $\psi$ and $\phi$ be the solution of
the problem (\ref{KGS}) with $d=\mu=\lambda=1$ and (\ref{numer example}),
which is obtained numerically on a bounded interval $\Omega=[-512,512]$
by the MTI-FP method
(\ref{MTI-FP S})-(\ref{MTI-FP E}) with very fine mesh $h=1/16$ and
time step $\tau=10^{-4}$.
Similarly, let $\psi_{\rm sw}$ and $z_{\rm sw}$ be the solution of
the problem (\ref{modellimit92}) with $d=\mu=\lambda=1$ and (\ref{schw45})
and (\ref{numer example}), which is obtained numerically on
the bounded interval $\Omega$
by the EWI-FP method \cite{Cai2} with very fine mesh $h=1/16$ and
time step $\tau=10^{-4}$; and let $\psi_{\rm s}$ and $z_{\rm s}$ be the solution of
the problem (\ref{modellimit4}) with $d=\mu=\lambda=1$ and (\ref{schint4})
and (\ref{numer example}), which is obtained numerically on
the bounded interval $\Omega$
by the TSFP method \cite{Baoreview} with very fine mesh $h=1/16$ and
time step $\tau=10^{-4}$. Denote
\[
\phi_{\rm sw}(x,t)=\fe^{it/\eps^2}z_{\rm sw}(x,t)+\fe^{-it/\eps^2}\overline{z_{\rm sw}}(x,t),\quad
\phi_{\rm s}(x,t)=\fe^{it/\eps^2}z_{\rm s}(x,t)+\fe^{-it/\eps^2}\overline{z_{\rm s}}(x,t),
\]
and define the error functions as
\be
\begin{split}
&\eta_{\rm sw}(t):=\|\phi(\cdot,t)-\phi_{\rm sw}(\cdot,t)\|_{H^1}+
\|\psi(\cdot,t)-\psi_{\rm sw}(\cdot,t)\|_{H^1},\\
&\eta_{\rm s}(t):=\|\phi(\cdot,t)-\phi_{\rm s}(\cdot,t)\|_{H^1}+
\|\psi(\cdot,t)-\psi_{\rm s}(\cdot,t)\|_{H^1}.
\end{split}
\ee

Fig. \ref{fig:convergence} depicts $\eta_{\rm sw}(t)$ and $\eta_{\rm s}(t)$ for
 different $\eps$. From Fig. \ref{fig:convergence} and additional similar
 results not shown here for brevity, when $\eps\to0$,
we can observe numerically the solution of
the KGS equations (\ref{KGS}) converges quadratically and uniformly in time
to that of its limiting model (\ref{modellimit92}) as
\be
\|\psi(\cdot,t)-\psi_{\rm sw}(\cdot,t)\|_{H^1(\Omega)}+
\|\phi(\cdot,t)-\phi_{\rm sw}(\cdot,t)\|_{H^1(\Omega)}\leq C_0\eps^2,\qquad t\ge0,
\ee
and respectively, to its limiting model (\ref{modellimit4}) as
\be
\|\psi(\cdot,t)-\psi_{\rm s}(\cdot,t)\|_{H^1(\Omega)}+\|\phi(\cdot,t)-\phi_{\rm s}
(\cdot,t)\|_{H^1(\Omega)}\leq (C_1+C_2T)\eps^2, \quad 0\le t\le T,
\ee
where $C_0$, $C_1$  and $ C_2>0$ are positive constants which are
independent of $\eps$ and $T$. Rigorous mathematical justification of
these numerical observations is on-going.

\begin{figure}[t!]
\centerline{\psfig{figure=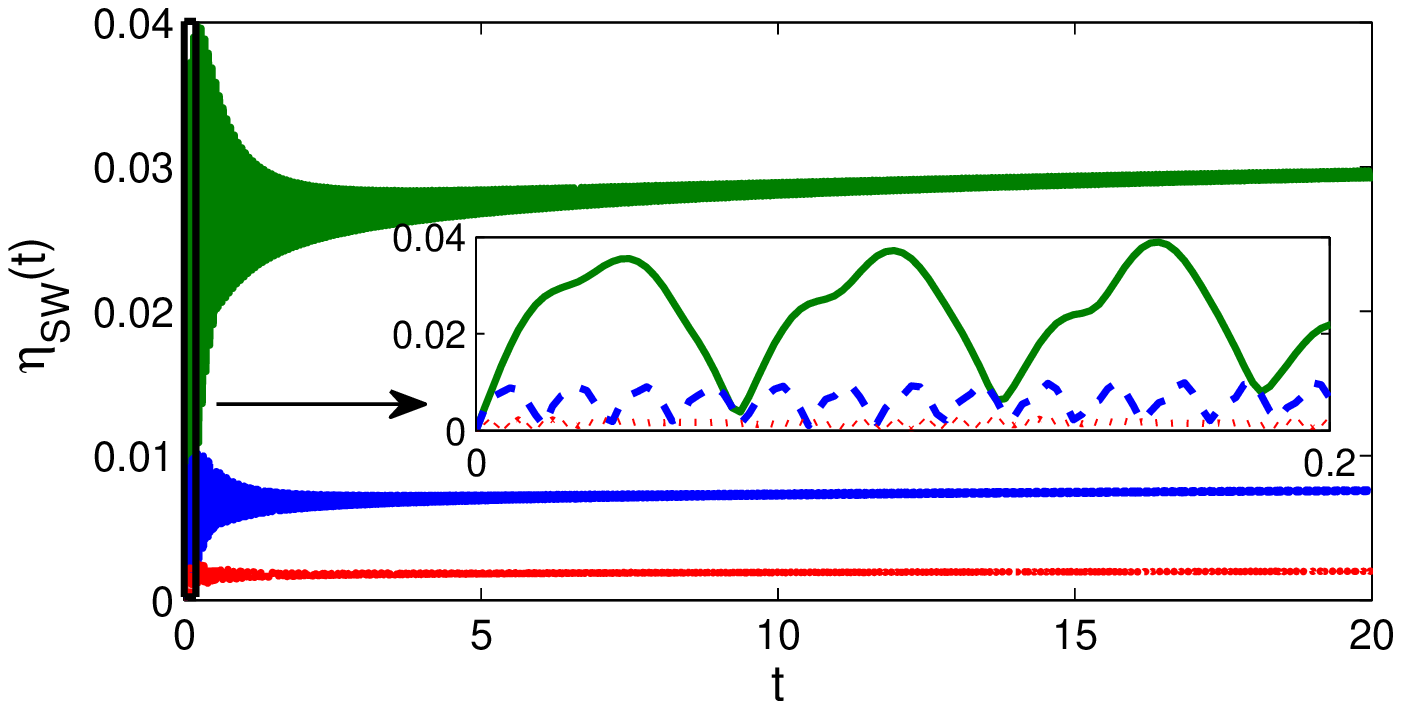,height=4.1cm,width=12cm}}
\centerline{\psfig{figure=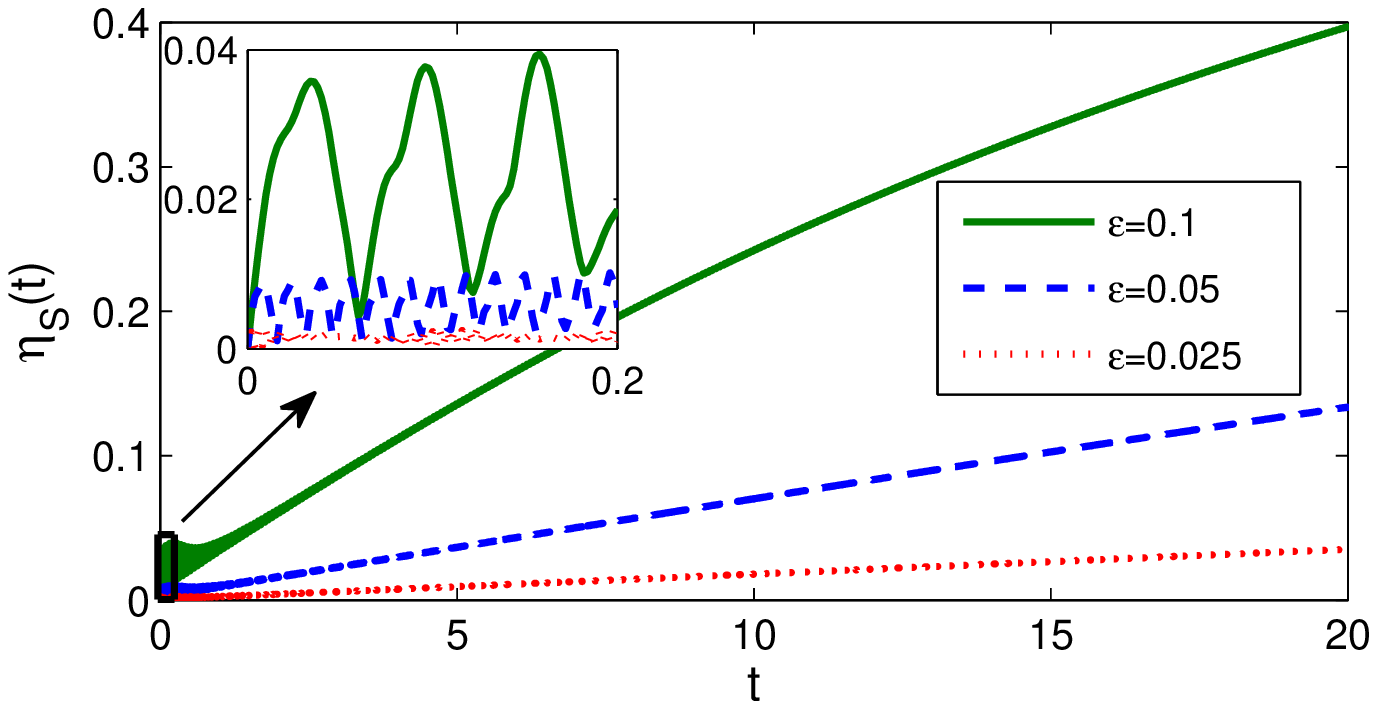,height=4.1cm,width=12cm}}
\centerline{\psfig{figure=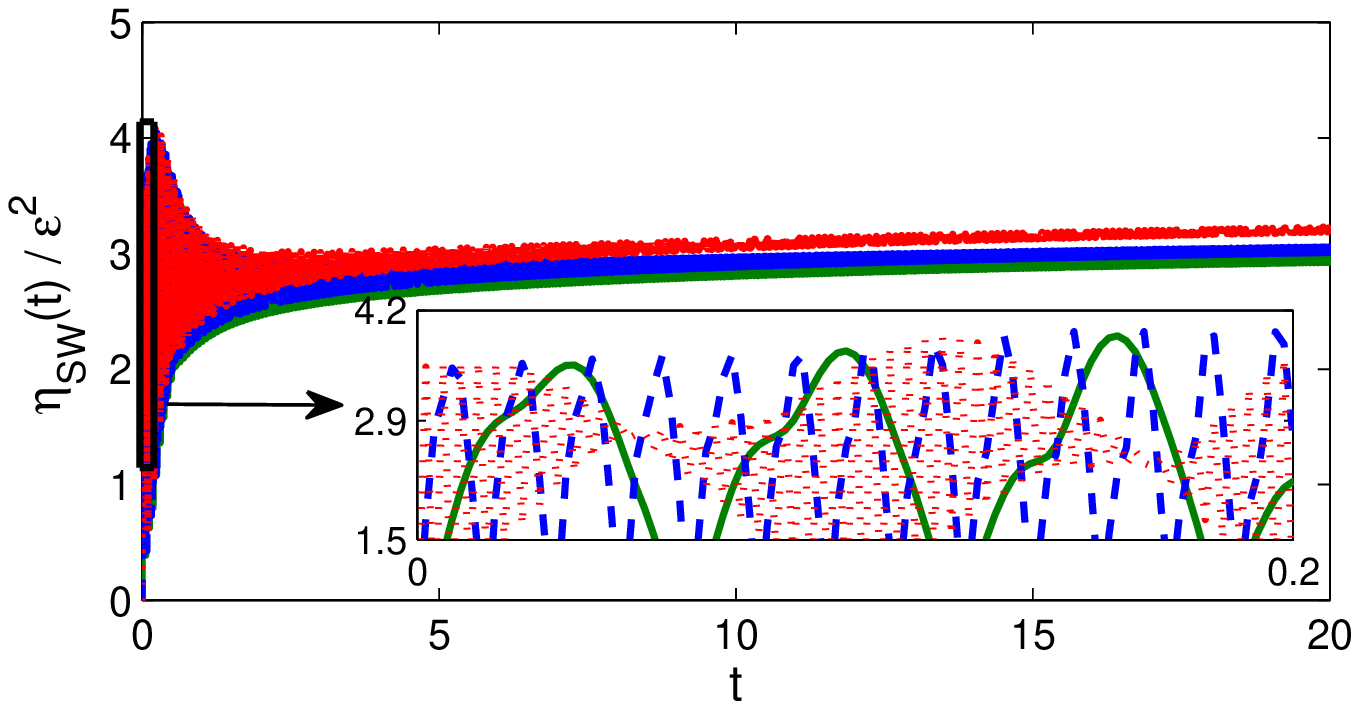,height=4.1cm,width=12cm}}
\centerline{\psfig{figure=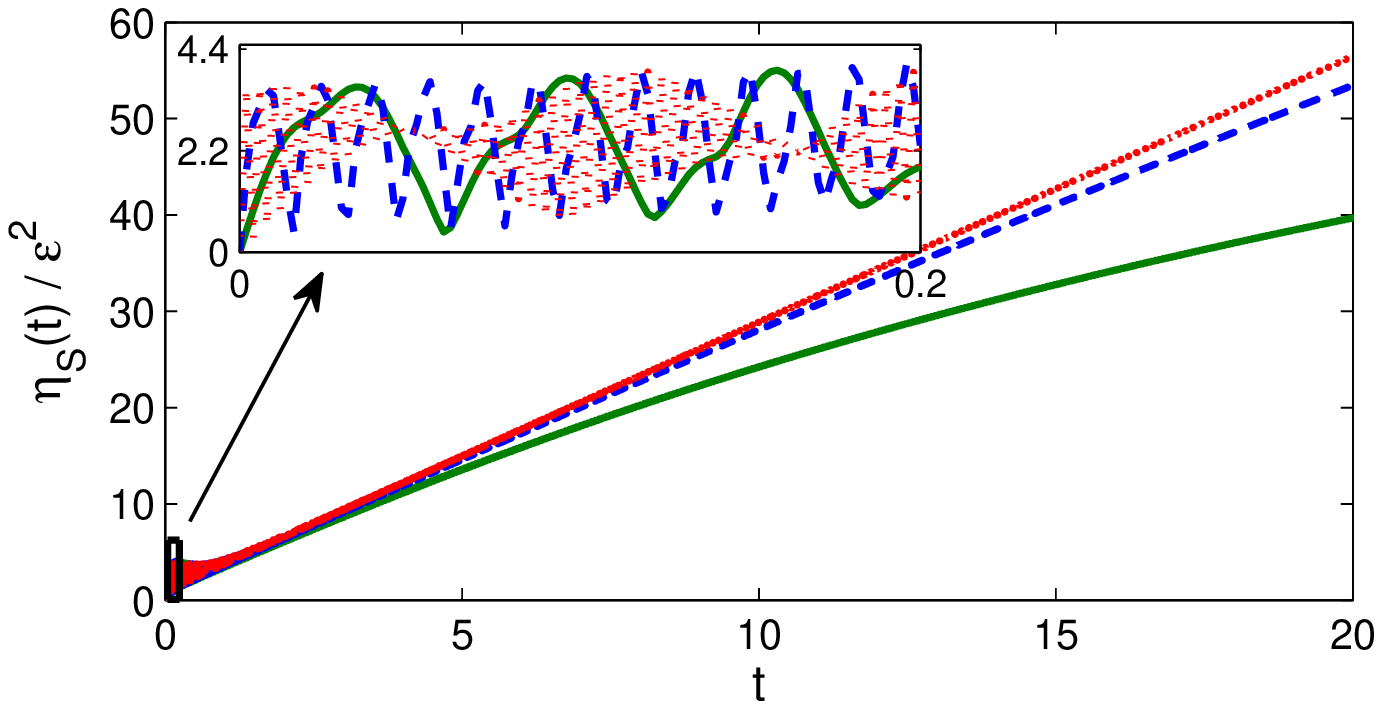,height=4.1cm,width=12cm}}
\caption{Time evolution of  $\eta_{\rm sw}(t)$ and $\eta_{\rm s}(t)$ for different $\eps$.}\label{fig:convergence}
\end{figure}

\section{Conclusion}\label{sec: conclusion}
A multiscale time integrator Fourier pseudospectral (MTI-FP) method
was proposed and analyzed for solving the Klein-Gordon-Schr\"{o}dinger (KGS)
equations with a dimensionless parameter $0<\eps\leq1$ which is inversely
proportional to the speed of light.
The key ideas for designing the MTI-FP method are based on
(i) carrying out a multiscale decomposition by frequency
at each time step with proper choice of transmission conditions between adjacent time intervals,
and (ii) adapting the Fourier spectral method for spatial discretization and
the EWI for integrating second-order highly oscillating ODEs.
Rigorous error bounds for the MTI-FP method were established,
which imply that the MTI-FP method converges uniformly and optimally
in space with spectral convergence rate, and uniformly in time with linear convergence rate
for $\eps\in(0,1]$ and optimally with quadratic convergence rate in the regimes when either $\eps= O(1)$ or $0<\eps\leq\tau$.
Numerical results confirmed these error bounds and showed that the MTI-FP method offers compelling advantages over
 classical methods in the nonrelativistic regime $0<\eps\ll1$.
 Numerical results suggest that the solution of the KGS equations
converges quadratically  to that of its limiting models when
$\eps\to0$.


\bigskip
\noindent {\bf Acknowledgements} This work was supported by
the Ministry of Education of Singapore grant R-146-000-196-112.
Part of this work was done when the authors were visiting the Institute for
Mathematical Science at the National University of Singapore in 2015.


\setcounter{equation}{0}
\renewcommand{\theequation}{A.\arabic{equation}}
\begin{center}
{\bf Appendix A. Explicit formulas for the coefficients used in (\ref{z r app})}
\end{center}
Define the functions
\begin{numcases}
\ p_l(s):=\int_0^s\frac{\sin(\omega_l(s-\theta))}{\eps^2\omega_l}d\theta,
\quad p_l'(s):=\int_0^s\frac{\cos(\omega_l(s-\theta))}{\eps^2}d\theta,\nonumber\\
q_l(s):=\int_0^s\frac{\sin(\omega_l(s-\theta))}{\eps^2\omega_l}\theta d\theta,\quad
q_l'(s):=\int_0^s\frac{\cos(\omega_l(s-\theta))}{\eps^2}\theta d\theta,\label{cldl765}\\
c_l^+(s):=i\fe^{-i\mu_l^2s}\int_0^s\fe^{i\left(\mu_l^2+\frac{1}{\eps^2}\right)\theta}d\theta,\quad
c_l^-(s):=i\fe^{-i\mu_l^2s}\int_0^s\fe^{i\left(\mu_l^2-\frac{1}{\eps^2}\right)\theta}d\theta,\nonumber\\
d_{l}^+(s):=i\fe^{-i\mu_l^2s}\int_0^s\fe^{i\left(\mu_l^2+\frac{1}{\eps^2}\right)\theta}\theta d\theta,\quad
d_l^-(s):=i\fe^{-i\mu_l^2s}\int_0^s\fe^{i\left(\mu_l^2-\frac{1}{\eps^2}\right)\theta}\theta d\theta.\nonumber
\end{numcases}
Taking $s=\tau$ in (\ref{cldl765}),
after a detailed computation, we have
\begin{align*}
&p_l(\tau)=\frac{1-\cos(\omega_l\tau)}{\eps^2\omega_l^2},
\quad p_l'(\tau)=\frac{\sin(\omega_l\tau)}{\eps^2\omega_l},\nonumber\\
&q_l(\tau)=\frac{\tau\omega_l-\sin(\omega_l\tau)}{\eps^2\omega_l^3},\quad
q_l'(\tau)=\frac{1-\cos(\omega_l\tau)}{\eps^2\omega_l^2},\nonumber\\
&c_l^+(\tau)=\frac{\eps^2\fe^{-i\mu_l^2\tau}\left(\fe^{i\tau(\mu_l^2+
\frac{1}{\eps^2})}-1\right)}{1+\eps^2\mu_l^2},\quad
c_l^-(\tau)=\frac{\eps^2\left(\fe^{-i\tau\mu_l^2}-\fe^{-i\frac{\tau}
{\eps^2}}\right)}{1-\eps^2\mu_l^2},\label{cldl765}\\
&d_{l}^+(\tau)=\frac{i\eps^2}{(1+\eps^2\mu_l^2)^2}
\left[\fe^{\frac{i\tau}{\eps^2}}\left(\eps^2-i\tau\left(1+\eps^2\mu_l^2\right)\right)
-\eps^2\fe^{-i\mu_l^2\tau}\right],\nonumber\\
&d_l^-(\tau)=\frac{i\eps^2}{(1-\eps^2\mu_l^2)^2}\left[\fe^{-\frac{i\tau}{\eps^2}}
\left(\eps^2-i\tau\left(\eps^2\mu_l^2-1\right)\right)-\eps^2\fe^{-i\mu_l^2\tau}\right].\nonumber
\end{align*}

\end{document}